\documentclass[12pt]{amsart}
\usepackage{amsfonts,graphicx, amsthm,amsmath,amssymb}

\usepackage{hyperref}
\usepackage{color}
\usepackage{enumerate}

\usepackage{todonotes}

\usepackage{bbold} 

\newcommand{\bbN}{\mathbb{N}}

\newcommand{\R}{\mathbb{R}}

\newcommand{\p}{\mathbb{P}}
\newcommand{\E}{\mathbb{E}}

\newcommand{\En}{\mathcal{E}}
\newcommand{\tEn}{\widetilde{\En}}

\newcommand{\bv}{\overline{v}}
\newcommand{\bg}{\overline{g}}
\newcommand{\bK}{\overline{K}}
\newcommand{\bA}{\overline{A}}

\newcommand{\tC}{\widetilde{C}}

\newcommand{\dn}{\delta_E}
\newcommand{\Cn}{C_E}

\newcommand{\dm}{m}

\newcommand{\mM}{\mathcal{M}}
\newcommand{\mK}{\mathcal{K}}

\newcommand{\Diff}{\mathrm{Diff}}
\newcommand{\Homeo}{\mathrm{Homeo}}
\newcommand{\Leb}{\mathrm{Leb}}
\newcommand{\supp}{\mathop{\mathrm{supp}}}
\newcommand{\Lip}{\mathop{\mathrm{Lip}}}
\newcommand{\fL}{\frak{L}}

\newcommand{\SL}{\mathrm{SL}}
\newcommand{\TV}{\mathrm{TV}}
\newcommand{\vol}{\mathrm{vol}}
\newcommand{\Jac}{\mathrm{Jac}}

\newcommand{\mgr}{\mu}
\newcommand{\msp}{\nu}

\newcommand{\mTn}{\theta}
\newcommand{\dens}{\rho}
\newcommand{\mT}{\Theta}
\newcommand{\mTt}{\widehat{\Theta}}
\newcommand{\mTtn}{\widehat{\theta}}
\newcommand{\mTf}{\Theta'}
\newcommand{\mTfn}{\theta'}

\newcommand{\Ind}{\mathrm{1\!\!I}}

\newcommand{\eps}{\varepsilon}
\newcommand{\const}{\mathrm{const}}
\newcommand{\diam}{\mathrm{diam}}

\newcommand{\bhref}[2]{\href{#1}{\textcolor{blue}{#2}}}

\newcommand{\arXiv}[1]{arXiv:\bhref{https://arxiv.org/pdf/#1.pdf}{#1}}

\newtheorem{theorem}{Theorem}[section]
\newtheorem*{theorem-non}{Theorem}
\newtheorem{lemma}[theorem]{Lemma}
\newtheorem{proposition}[theorem]{Proposition}
\newtheorem{corollary}[theorem]{Corollary}

\theoremstyle{definition}
\newtheorem{definition}[theorem]{Definition}
\newtheorem{remark}[theorem]{Remark}

\numberwithin{equation}{section}

\begin{document}

\title{H\"older regularity of stationary measures}
\date{}
\author{Anton Gorodetski}

\address{Anton Gorodetski, Department of Mathematics, University of California, Irvine, CA~92697, USA}

\email{asgor@uci.edu}

\thanks{A.\ G.\ and G.\ M.\ were supported in part by NSF grant DMS--1855541.}

\author{Victor Kleptsyn}

\address{Victor Kleptsyn, Univ Rennes, CNRS, IRMAR - UMR 6625, F-35000 Rennes, France}

\email{victor.kleptsyn@univ-rennes1.fr}

\thanks{V.K. was supported in part by ANR Gromeov (ANR-19-CE40-0007) and by Centre Henri Lebesgue (ANR-11-LABX-0020-01)
}

\author{Grigorii Monakov}

\address{Grigorii Monakov, Department of Mathematics, University of California, Irvine, CA~92697, USA}

\email{gmonakov@uci.edu}


\begin{abstract}
We consider smooth random dynamical systems defined by a distribution with a finite moment of the norm of the differential, and prove that under suitable non-degeneracy conditions any stationary measure must be H\"older continuous. The result is a vast generalization of the classical statement on H\"older continuity of stationary measures of random walks on linear groups.
\end{abstract}

\maketitle


\section{Introduction}\label{s.intro}

Let $M$ be a smooth closed Riemannian  manifold, 
and $\mgr$ be a Borel probability measure on $\Diff^1(M)$, the set of $C^1$-diffeomorphisms of~$M$. 
Consider the corresponding random dynamical system, given by the compositions \[
T_{n}:=f_{n}\circ\dots\circ f_1,
\]
where $f_i\in \Diff^1(M)$ are chosen randomly and independently, with respect to the distribution $\mgr$. 

If an initial point $x\in M$ is distributed with respect to a probability measure $\msp$, 
 one can consider the distribution $\mgr*\msp$ of its random image $f(x)$. In other words, $\mgr*\msp$ is the $\mgr$-averaged push-forward image of the measure $\msp$:
\[
\mgr*\msp := \int_{\Diff^1(M)} (f_* \msp) \, d\mgr(f).
\]


The measure $\msp$ is called \emph{$\mgr$-stationary} if $\mgr*\msp=\msp$. Stationary measures of random dynamical systems are analogues of invariant measures of  deterministic maps, and their properties are of crucial importance for many results in random dynamics, see \cite{A, BH, BQ16, BL, Fu, Fu1, Kif1, Kif2, LQ, M} and references therein. 

For a diffeomorphism $f \in \Diff^1(M)$  denote
    \begin{equation*}
        \fL (f)  = \max(\Lip(f), \Lip(f^{-1})).
    \end{equation*}
Here is the main result of this paper:
\begin{theorem}\label{t.main}
Suppose that $\mgr$ is a probability distribution on $\Diff^1(M)$ such that $\int\frak{L}(f)^\gamma d\mgr(f)<\infty$ for some $\gamma>0$. Suppose also that there is no probability measure $m$ on the manifold $M$ invariant under every map $f\in \text{supp}\,\mgr$. Then every stationary measure of a random dynamical system defined by the distribution $\mgr$ is H\"older continuous.
\end{theorem}

In Section \ref{s.2} below we provide formal definitions and more general and stronger versions of this result. 
 Specifically, in Theorem \ref{t:main-2} we show that averaged images of any initial measure become H\"older regular on any scale larger than some threshold, which decays exponentially in the number of iterates. Theorem \ref{t:main-3} is a similar result for a non-stationary case, namely, it claims that, informally speaking,  by averaging with respect to a different distribution on each step, one ``regularizes''  a probability measure on a manifold, bringing it closer to the subspace of H\"older measures. 

 Let us now discuss two classical statements (on H\"older continuity of stationary measures for random matrix products and for iterated function systems) that can be considered partial cases of Theorem~\ref{t.main}.

\subsection{Stationary measures for random matrix products}\label{ss.smrmp}

Consider random products of iid matrices $A_nA_{n-1}\ldots A_1$, where each $A_i$ is a random matrix chosen with respect to a probability distribution $\tilde \mgr$ on $SL(d, \mathbb{R})$, $d\ge2$.
Let $S_{\tilde \mgr}$ be the closed semigroup in $SL(d, \mathbb{R})$ generated by matrices from $\text{supp}\,{\tilde\mgr}$. Then $S_{\tilde \mgr}$ is called {\it strongly irreducible} if there is no finite family of proper non-zero subspaces $V_1, V_2, \ldots, V_N\subset \mathbb{R}^d$ such that
$$
A(\cup_{i=1}^N V_i)=\cup_{i=1}^N V_i \ \text{for all}\ A\in S_{\tilde \mgr},
$$
and {\it proximal} (or {\it contracting}) if there exists a sequence $A_1, A_2, \ldots, \in S_{\tilde \mgr}$ and a sequence of real numbers $a_1, a_2, \ldots$ such that the sequence of operators $\{a_iA_i\}$ converges in norm to a linear endomorphism of $\mathbb{R}^d$ of rank one, see \cite{BL} for a detailed discussion of these notions.

Every $A\in \SL(d, \mathbb{R})$ induces a projective map $f:\mathbb{RP}^{d-1}\to \mathbb{RP}^{d-1}$. Let us denote by $\mgr$ the distribution on projective maps induced by the distribution $\tilde \mgr$ on $\SL(d, \mathbb{R})$.
\begin{theorem}[Guivarc’h, \cite{G}]\label{t.g}
Suppose that, in the setting above, $S_{\tilde \mgr}$ is strongly irreducible and proximal, and
$$
\int\|A\|^{\gamma}d\tilde \mgr<\infty
$$
for some $\gamma>0$. Then the random dynamical system given by the distribution $\mgr$ on the projective maps of $\mathbb{RP}^{d-1}$ has unique stationary measure $\msp$ on $\mathbb{RP}^{d-1}$, and $\msp$ is H\"older continuous.
\end{theorem}
An alternative proof of Theorem \ref{t.g} (inspired by \cite{BFLM}) can be found in \cite{BQ16}. A proof of H\"older continuity of stationary measures under somewhat weaker assumptions was obtained in \cite{AG}.

Notice that Theorem \ref{t.g} can be considered as a partial case of Theorem \ref{t.main}. Indeed, strong irreducibility and proximality conditions imply absence of common invariant measure for the corresponding projective maps (this follows, for example, from Propositions 1.6 and 2.3 in \cite[Chapter III]{BL}), and finiteness of momentum condition implies the corresponding condition in Theorem \ref{t.main}.

Notice that positivity of Lyapunov exponents for random matrix products due to Furstenberg Theorem holds under weaker conditions.  In particular, instead of finiteness of the integral $\int\|A\|^{\gamma}d\tilde \mgr$ it is enough to require that $\int\log\|A\|d\tilde \mgr<\infty$. It is interesting to notice that it is not sufficient to ensure H\"older continuity of the stationary measure, see Appendix \ref{ss.example} below. At the same time, weaker statements on moduli of continuity can be established under these weaker assumptions, see Proposition 4.5 in \cite{BQ15} and Theorem 1.4 in \cite{DKW}.

Other properties of the stationary measure in the setting of Theorem \ref{t.g} were studied, sometimes under additional assumptions on the distribution $\tilde \mgr$. For example, in \cite{KL} it was conjectured that if the distribution $\tilde \mgr$ in Theorem \ref{t.g} is finitely supported, then the stationary measure must be singular. A counterexample to that conjecture with absolutely continues stationary measure was constructed in \cite{BPS}, and with stationary measure having smooth density -- in \cite {B}. A related question is whether in the case of a distribution $\tilde \mgr$ that generates a Fuchsian group the corresponding stationary measure must be singular. It was settled for non-cocompact actions by Guivarc'h and Le Jan in \cite{GL} (see also \cite{DKN} and \cite{GMT}), but is still open in full generality.


\subsection{Stationary measures for iterated function systems}

An iterated function system is defined by a finite collection of contractions of a complete metric space, usually of $\mathbb{R}^d$ or a compact subset of $\mathbb{R}^d$. In \cite[Proposition 2.2]{FL} it is shown that in the case of an iterated function system generated by a finite collection of similarities in $\mathbb{R}^d$ (without a common fixed point) applied with given prescribed probabilities the unique stationary measure must be H\"older continuous. Let us generalize this statement by allowing infinite families of contractions.
\begin{theorem}\label{t.IFS}
Let $B\subset \mathbb{R}^d$ be a closed ball, and $\mgr$ be a probability distribution on the space of all $C^1$ maps (contractions) $f:B\to B$ such that $f$ is a diffeomorphims of $B$ onto $f(B)$, and $\max_{x\in B}\|Df(x)\|<1$. Suppose that
$$
\int \left[\Lip(f^{-1}|_{f(B)})\right]^\gamma d\mgr<\infty \ \text{for some}\ \gamma>0,
$$
and there is no common fixed point for all maps $f\in \text{supp}\, \mgr$. Then the unique stationary measure of the random dynamical system generated by the distribution $\mgr$ is H\"older continuous.
\end{theorem}

Uniqueness of a stationary measure in this context is well known, e.g. see \cite{Hu}  or \cite{Sz}. H\"older continuity of the stationary measure in Theorem~\ref{t.IFS} is an immediate consequence of Theorem \ref{t.main}. Indeed, due to the contraction mapping principle, any contraction must have exactly one invariant measure, namely an atomic measure at a fixed point. Therefore, existence of a common invariant measure implies existence of a common fixed point. Since Theorem \ref{t.main} can be applied to the maps of a manifold with boundary (see Remark \ref{rem:boundary} below), Theorem \ref{t.IFS} follows.

We believe that Theorem \ref{t.IFS} is well known to the experts, but to the best of our knowledge it has not appeared in print in this generality. 

A famous specific example of one-dimensional stationary measures are those given by the Bernoulli convolutions. These are stationary measures for linearly contracting self-maps of the interval $[0,1]$,
\[
x\mapsto \lambda x \, \text{ and }\, x\mapsto 1-\lambda(1-x),
\]
where $\lambda\in[\frac{1}{2},1]$. These measures were extensively studied for at least 70 years; on the one hand, it is known \cite{E} that for some special values of $\lambda$ (for instance, for inverse golden ratio) the stationary measure $\msp_{\lambda}$ is singular and, moreover, has its Hausdorff dimension strictly less than~$1$. On the other hand, a famous result by Solomyak \cite{Sol} establishes absolute continuity of $\msp_{\lambda}$ for almost every $\lambda\in [\frac{1}{2},1]$. Moreover, due to an improvement by Shmerkin \cite{Shm} the set of exceptional values of the parameter $\lambda$ is in fact of zero Hausdorff dimension. For general surveys we refer the reader to \cite{PSS, Va, Va2}. For recent results on a lower bound on the Hausdorff dimension of~$\msp_{\lambda}$  see \cite{FF, KPV, Va3}.


Questions about exact dimensionality of stationary measures of iterated function systems and explicit formulas for their dimension (usually of Ledrappier-Young type, i.e. in terms of entropy and Lyapunov exponents) were heavily studied, e.g. see \cite{BK, F, FJ, Ho, HR}, and references there. Similar questions on stationary measures of random matrix products were studied in \cite{HS, Led, LL, R}. Notice that while any H\"older continuous measure must have positive Hausdorff dimension, without any additional assumptions it does not have to be exact dimensional, and, other way around, exact dimensional measure does not necessarily have to be H\"older continuous.


\section{Definitions and main results}\label{s.2}

\begin{definition}
    We say that the measure $\msp$ on $M$ is $(\alpha,C)$-H\"older, if
    \begin{equation}\label{eq:Holder-a-C}
        \forall x\in M \quad \forall r>0 \quad \msp(B_r(x))<C r^{\alpha}.
    \end{equation}
\end{definition}

\begin{definition}\label{d.Lip}
    For a diffeomorphism $f \in \Diff^1(M)$ we define
    \begin{equation*}
        \Lip(f) = \sup_{x, y \in M} \left( \frac{d(f(x), f(y))}{d(x, y)} \right)
    \end{equation*}
    and
    \begin{equation*}
        \fL (f)  = \max(\Lip(f), \Lip(f^{-1})).
    \end{equation*}
\end{definition}

Our first main result claims 
 that  H\"older regularity of stationary measures is completely abundant. The statement below is a more formal and detailed version of Theorem \ref{t.main}.
\begin{theorem}\label{t:main-1}
    Let $\mgr$ be a probability measure on $\Diff^1(M)$, satisfying the following assumptions:
    \begin{itemize}
        \item \textbf{(finite moment condition)} There exists $\gamma > 0$ and $ C_0>0$
        such that
        \begin{equation} \label{PosMomentCond}
            \int \fL(f)^{\gamma} \, d\mgr(f) < C_0;
        \end{equation}
        \item \textbf{(no common invariant measure)} There is no finite measure $m$ on $M$ such that $f_*m=m$ for $\mgr$-a.e.~$f$.
    \end{itemize}
    Then there exist $\alpha>0$ and $C$ such that any $\mgr$-stationary probability measure $\msp$ on $M$ is $(\alpha,C)$-H\"older.
\end{theorem}
 \begin{remark}
 Denote by $\mM$ the space of all probability measures on $\Diff^1(M)$ equipped with weak-$\ast$ topology. If $\mK\subset \mM$ is a compact set, such that the assumptions of Theorem \ref{t:main-1} (with uniform $\gamma>0$ and $C_0>0$) hold for all $\mgr\in\mK$, then $\alpha$ and $C$ in the conclusion can be chosen uniformly in  $\mgr\in\mK$.
 \end{remark}
 \begin{remark}
Notice that the condition on absence of common invariant measures cannot be 
dropped without adding some other assumptions on a model. Indeed, otherwise it could happen that all maps from $\text{supp}\,\mgr$ preserve the same non-H\"older measure, e.g. some atomic measure. 
 \end{remark}

Theorem~\ref{t:main-1} states that all the stationary measures for a random dynamical system with no common invariant measure are uniformly H\"older. Thus, it is natural to ask if (averaged) iterations $\mgr^{*n}*\msp_0$ of a given non-stationary initial measure $\msp_0$ become ``more and more H\"older'' as the number $n$ of iterations increases.

However, if both initial measure $\msp_0$ and the measure $\mgr$ are atomic (for instance, if $\msp_0$ is a Dirac measure, and the measure $\mgr$ is supported on a finite number of diffeomorphisms), then any of these iterated images has atoms, and thus cannot be H\"older. 
Hence, one can only hope (and expect) the H\"older property on not-too-small scales. And indeed, it turns out to be the case: after $n$ iterations one has H\"older property on all scales above an exponentially small one. Namely, we have the following theorem, our second main result:
\begin{theorem}\label{t:main-2}
Assume that $\mgr$ satisfies the assumptions of Theorem~\ref{t:main-1}. Then there exist
$\alpha>0$, $C$ and $\kappa<1$ such that for any initial measure $\msp_0$, any number of iterations $n\in \bbN$, and any $x\in M$ one has:
\begin{equation}\label{e.scale1}
\text{if}\ \  r>\kappa^n \ \ \text{then}\ \ (\mgr^{*n}*\msp_0)(B_r(x)) < C r^{\alpha}.
\end{equation}
\end{theorem}
\begin{remark}
Equivalently, the conclusion (\ref{e.scale1}) in Theorem \ref{t:main-2} can be replaced by the following one (perhaps, with different values of the constants):
$$
\text{\it for any}\ \  r>0 \ \ \text{\it we have}\ \ (\mgr^{*n}*\msp_0)(B_r(x)) < C (r^{\alpha} + \kappa^n).
$$
\end{remark}

  Notice that  Theorem \ref{t:main-1} is an immediate consequence of~Theorem~\ref{t:main-2}. Indeed, if $\msp_0$ is a stationary measure, then $\mgr^{*n} * \msp_0 = \msp_0$, and since $\kappa^n \to 0$ as ${n \to \infty}$, applying Theorem \ref{t:main-2} to the measure $\msp_0$ shows that $\msp_0$ is H\"older continuous, and hence Theorem \ref{t:main-1} follows.

Finally, in some situations one has to consider non-stationary random dynamical systems, where the maps $f_i$ applied on different steps are chosen with respect to different distributions $\mgr_i$. One example of such situation is non-stationary version of the Furstenberg theorem \cite{GK, Go}, related to the Anderson Localisation for random Schr\"odinger operators in $l^2(\mathbb{Z})$, usually referred to as {\it Anderson Model}, in presence of a non-constant background potential.


Our third theorem states that in the non-stationary setting, under natural genericity assumptions the random (averaged) image of any given probability measure on $M$ after $n$ iterations satisfies the H\"older property on not-too-small scales.
\begin{theorem} \label{t:main-3}
    Let $\mK\subset \mM$ be a compact, satisfying the following conditions:
    \begin{itemize}
        \item \textbf{(uniform finite moment condition)} There exists $\gamma>0$, $C_0$ such that for any $\mgr\in\mK$ one has
        \begin{equation} \label{PosMomentCondComp}
            \int \fL(f)^\gamma \, d \mgr(f) < C_0;
        \end{equation}
        \item \textbf{(no deterministic images)} For any $\mgr\in\mK$ there are no probability measures $\msp,\msp'$ on~$M$ such that $f_* \msp = \msp'$ for $\mgr$-almost all $f\in \Diff^1(M)$.
    \end{itemize}
    Then there exist $\alpha>0$, $C$, $\kappa<1$ such that for any initial measure $\msp_0$, any $n$, and any distributions $\mgr_1,\dots, \mgr_n\in \mK$ the $n$-th image of $\msp_0$ satisfies $(\alpha,C)$-H\"older property on the scales up to $\kappa^n$:
    \[
        \forall x\in M \quad \forall r>\kappa^n \quad (\mgr_n * \dots * \mgr_1 * \msp_0)(B_r(x)) < Cr^{\alpha}.
    \]
\end{theorem}

\begin{remark} \label{rem:boundary}
    Theorems \ref{t:main-1}, \ref{t:main-2} and \ref{t:main-3} also hold for the case of $M$ being a compact manifold with a boundary. In this case we consider $\mgr$ that is a probability measure on the space of all diffeomorphisms of $M$ on the image.
\end{remark}

\begin{remark}
    While we formulate the results for $C^1$-diffeomorphisms, smoothness is not really used in the proof. We only need the maps to be bi-Lipschitz, so we notice here that 
    Theorems \ref{t:main-1}, \ref{t:main-2} and \ref{t:main-3} also hold for random dynamical systems defined by bi-Lipschitz maps, for probability measure $\mgr$ defined on the space of bi-Lipschitz homeomorphisms: $\{f \in \Homeo(M) \mid \fL(f) < \infty\}$.
\end{remark}

\subsection{``No invariant measure'' vs ``no deterministic images'' conditions}

The ``no deterministic images'' assumption in Theorem \ref{t:main-3} is a natural generalisation to the non-stationary setting of the ``no common invariant measures'' one. Indeed, in the non-stationary setting it does not make sense to directly compare measures before and after applying a random map, since by taking any (bounded) family of maps $g_i$ and considering a family of conjugations 
$f_i\mapsto \tilde{f_i}:= g_i \circ f_i \circ g_{i-1}^{-1}$ one can turn 
random orbits $x_i=f_i(x_{i-1})$ into 
 random orbits $y_i=\tilde{f_i}(y_i)$ given by $y_i=g_i(x_i)$, not changing any essential properties of the non-stationary random dynamical system, but destroying a common invariant measure.

Nevertheless, notice that in the stationary case, when both of these conditions make sense, ``no common invariant measure'' and ``no measure with deterministic images'' are essentially different: 
absence of one-step deterministic images is a strictly stronger assumption than absence of a common invariant measure. Indeed, for instance, one can take two diffeomorphisms $f,g$ of the circle, where $f$ is a North-South map with two fixed points (so the only invariant measures are concentrated at these points), and $g$ is a sufficiently small irrational rotation (thus having only one invariant measure, namely Lebesgue measure on the circle). Thus, $f$ and $g$ have no common invariant measure. However, the composition $g^{-1}\circ f$ (that is close to $f$) has a fixed point $x_0$ (near a fixed point of the map $f$). The relation $g^{-1}\circ f(x_0)=x_0$ implies that $f(x_0)=g(x_0)$, and thus the Dirac measure at $x_0$ has a deterministic image.

However, being able to compress several steps into a single one makes these two assumptions equivalent for their usage in the above theorems:
\begin{proposition}\label{p:k}
Let $\mgr\in\mM$ be a probability measure on $\Diff^1(M)$ such that there is no common invariant measure for all $f\in \text{supp}\,\mgr$. Then there exists $k\in \mathbb{N}$ such that $\mgr^{*k}$ satisfies ``no deterministic images'' condition (i.e. there are no probability measures $\msp,\msp'$ on~$M$ such that $f_* \msp = \msp'$ with $f=f_k\circ\ldots \circ f_1$ for $\mgr\times \mgr\times \ldots \times \mgr$-almost all $(f_1, f_2, \ldots, f_k)\in \left(\Diff^1(M)\right)^k$).
\end{proposition}

\begin{proof}[Proof of Proposition~\ref{p:k}]
Assume that for arbitrarily large $k$ there exist measures $\msp_k, \msp'_k$ such that $f_* \msp_k = \msp'_k$ for $\mgr^{*k}$-a.e.~$f\in\Diff^1(M)$. For any such $k$ consider the intermediate images (note that intermediate images must be deterministic, too):
\[
\msp_{k,0}:=\msp_k, \quad \msp_{k,j}=f_*\msp_{k,j-1}, \quad j=1,\dots, k, \quad \text{for $\mgr$-a.e. $f\in\Diff^1(M)$}.
\]
Let $\bar{\msp}_{k}$ be the time averages of these measures:
\[
\bar{\msp}_{k}:=\frac{1}{k}\sum_{j=0}^{k-1} \msp_{k,j}.
\]
Then, due to the standard time-averaging (Krylov-Bogolyubov) argument any accumulation point~$\bar{\msp}$ of the measures $\bar{\msp}_k$ is a common invariant measure for $\mgr$-a.e. $f\in \Diff^1(M)$. Namely, for $\mgr$-a.e. $f\in\Diff^1(M)$ one has
\[
f_* \bar{\msp}_k = \msp_k + \frac{1}{k}(f_* \msp_{k,k-1} -\msp_{k,0}),
\]
providing $f_* \bar{\msp}=\msp$ after passing to a limit point $\bar{\msp}$. 
This leads to a contradiction with the assumption of the proposition.
\end{proof}

Notice that Proposition \ref{p:k} implies that it suffices to prove Theorems \ref{t:main-1} and \ref{t:main-2} under a more restrictive ``no deterministic images'' assumption instead of the ``no common invariant measures'' condition. In particular, Theorem~\ref{t:main-2} can be considered a corollary of Theorem~\ref{t:main-3}.

\section{Ideas of the proof}\label{s.ideas}

In this section we informally present the ideas and motivations behind the proofs of our main results, before passing to the formal proofs in Section~\ref{s:proofs}.


We start with slightly weaker, but less technically complicated arguments, that allow us to establish the \emph{existence} of a H\"older stationary measure.

As the first step, for any $\alpha>0$ and a  measure $\msp$ on $M$ define its \emph{energy} as 
\begin{equation}\label{eq:En-def}
\En_{\alpha}(\msp):= \iint_M d(x,y)^{-\alpha} d\msp(x) d\msp(y)
\end{equation}

An immediate application of Markov's inequality shows that when $\En_{\alpha}(\msp)$ is finite, the measure is $\frac{\alpha}{2}$-H\"older, and the H\"older constant can be estimated in terms of the energy. Indeed, for any $x\in M$ and $r>0$ one has
\begin{equation}\label{eq:Markov}
\En_{\alpha}(\msp) \ge (2r)^{-\alpha} \msp(B_r(x))^2,
\end{equation}
as pairwise distances between the points of $B_{r}(x)$ do not exceed~$2r$. Hence, to establish the H\"older property of a measure it would suffice to show that its energy is finite.

The key step here is to show that for a sufficiently small $\alpha$, convolution with $\mgr$ makes sufficiently high energies decrease. Namely, there exists $\alpha_0>0$ and constants $C_0>0$, $0<\lambda<1$ such that for any $\alpha<\alpha_0$ one has
\begin{equation}\label{eq:En-decr}
\En_{\alpha}(\mgr*\msp) \le \lambda \En_{\alpha}(\msp) + C_0.
\end{equation}

To do so, we re-interpret the energy~\eqref{eq:En-def} in terms similar to $L_2$-norm of a convenient  function. Namely, for a given measure~$\msp$, consider the function
\begin{equation}\label{eq:theta-def}
\dens_{\alpha}[\msp](y):=\int_M \varphi_{\alpha}(d(x,y)) \, d\msp(x),
\end{equation}
where $\varphi_{\alpha}(r):=r^{-\frac{k+\alpha}{2}}$
and $k$ is the dimension of the manifold~$M$. Take the square of the $L_2$-norm (with respect to the Lebesgue 
measure 
 on~$M$) of this function,
\begin{equation}\label{eq:tEn-def}
\tEn_{\alpha}(\msp):=\int_M \dens_{\alpha}[\msp](y)^2 \, d\Leb(y).
\end{equation}
It turns out that there exists constant $c_{\alpha}$ 
such that for any $\msp$
\begin{equation}\label{eq:En-tEn}
 \tEn_{\alpha}(\msp) =(c_{\alpha} + o(1)) \cdot \En_{\alpha}(\msp),
\end{equation}
where $o(1)$ tends to zero as either of energies $\En_{\alpha}$, $\tEn_{\alpha}$ tend to infinity.
Indeed, the right hand side of~\eqref{eq:tEn-def} can be rewritten as
\begin{multline}\label{eq:tEn-kernel}
\tEn_{\alpha}(\msp) = \int_M \dens_{\alpha}^2[\msp](y) \, d\Leb(y)
\\ =
\iiint_{M\times M\times M} \varphi_{\alpha}(d(x,y)) \varphi_{\alpha}(d(z,y)) \, d\msp(x) d\msp(z) \, d\Leb(y)
\\ =
\iint_{M\times M} K_{\alpha}(x,z) d\msp(x) d\msp(z),
\end{multline}
where
\begin{equation}\label{eq:K-alpha}
K_{\alpha}(x,z) := \int_M \varphi_{\alpha}(d(x,y)) \varphi_{\alpha}(d(z,y)) \, d\Leb(y).
\end{equation}
If $M$ was replaced by the Euclidean space $\R^k$, the kernel~\eqref{eq:K-alpha} would have exactly the form $c_{\alpha} |x-z|^{-\alpha}$, where $c_{\alpha}$ is a constant, due to the symmetry and scaling arguments ($K_{\alpha}(x,z)$ can depend only on the distance between $x$ and $z$ and should scale as its $\alpha$'s inverse power). Thus, on $\R^k$ one would have exact proportionality $c_{\alpha} \En_{\alpha}= \tEn_{\alpha}$. Meanwhile, for a general compact manifold $M$ high values of energy $\En$ can come only from points close to each other, and on a small scale a Riemannian manifold is almost Euclidean.

Finally, the key argument goes by contradiction. Namely, for a small~$\alpha$ the energy $\En_{\alpha}$ is almost unchanged by applying a given diffeomorphism~$f$, as the distances are changed at most $\fL(f)$ times, and their~$\alpha$'s powers are changed by a factor close to~$1$. Due to~\eqref{eq:En-tEn}, this implies that for a large value of energy the same applies to~$\tEn_{\alpha}$.

However, the latter energy admits an $L_2$-norm interpretation: the definition (\ref{eq:tEn-def}) is the squared length $\langle \rho, \rho \rangle_{L_2(M, \Leb)}$ of a vector (function) $\rho\in L_2(M)$, linearly associated to the measure $\msp$.  In particular, if the $L_2$-norm of the average of images~$f_* \msp$ (where $f$ is distributed w.r.t.~$\mgr$) does not decrease, it means that (most of) these vectors are essentially (almost) aligned. Considering the non-probability measure
\begin{equation}\label{eq:Theta-m-def}
\mT_{\alpha}[\msp]:= \dens_{\alpha}[\msp]^2(y) \, dLeb(y),
\end{equation}
we see that the (almost) alignment of the vectors implies that the corresponding normalized (probability) measure
$\mTn_{\alpha}[\msp]:=\frac{1}{\tEn_{\alpha}(\msp)} \mT_{\alpha}[\msp]$ on~$M$
is a measure with (almost) deterministic image. Passing to the limit as~$\alpha$ tends to~0, we find a measure with the deterministic image, thus obtaining a contradiction.

The relation~\eqref{eq:En-decr} that we have (non-rigorously) established implies that when the energy $\En_{\alpha}(\msp)$ is large, the passage to the averaged image $\mgr*\msp$ decreases it. This already suffices to establish the \emph{existence} of a H\"older stationary measure. Indeed, take any measure $\msp_0$ of finite $\alpha$-energy (for instance, the Lebesgue measure on $M$). Consider the sequence of its averaged images,
\[
\msp_n:=\mgr*\msp_{n-1}, \quad n=1,2,\dots;
\]
their $\En_{\alpha}$ energies then stay uniformly bounded. Hence, the same bound for the energy applies to their Cesaro averages $\bar{\msp}_n:=\frac{1}{n} \sum_{j=0}^{n-1} \msp_j$. Finally, any accumulation point of this  latter sequence is necessarily a stationary measure, is of finite energy, and, hence, H\"older regular. This completes the (informal sketch of the) proof of the existence of a H\"older stationary measure.

Now, for a stationary measure $\msp$ the inequality~\eqref{eq:En-decr} reads as
\[
\En_{\alpha}(\msp) \le \lambda \En_{\alpha}(\msp) + C_0,
\]
thus implying that either the energy $\En_{\alpha}(\msp)\le \frac{C_0}{1-\lambda}$, or that this energy is infinite. Unfortunately, the above arguments do not exclude the latter possibility, though suggest that it should not take place: large energies ``decrease at infinity''. Also, if one takes a Dirac measure as an initial one, 
an averaging with respect to a distribution that has atomic component will always have some atoms, and hence will never be H\"older regular. So in order to prove Theorems~\ref{t:main-1}--\ref{t:main-3}, definitions of energy in (\ref{eq:En-def}) and (\ref{eq:tEn-def}) are to be modified to exclude the infinity from the list of possibilities.


Such a modification is done by replacing the interaction kernel $d(x,y)^{-\alpha}$ in the definition of the energy $\En_{\alpha}$ by a bounded function, choosing a distance~$\eps$ at which the interaction in the energy is cut-off.
To do so, we take (see Definition~\ref{d:En-eps} below)
\[
\En_{\alpha,\eps}(\msp):= \iint_M U_{\alpha,\eps}(d(x,y)) d\msp(x) d\msp(y).
\]
for some bounded interaction potential $U_{\alpha,\eps}(r)$.
Unfortunately, bluntly taking $\left[\max(r,\eps)\right]^{-\alpha}$ in the role of $U_{\alpha,\eps}(r)$ makes it harder to work with the $L_2$-interpretation. 
 Thus we first modify the $L_2$-definition, applying a cut-off in the construction of the function $\rho$, and then use it to define the energy.

\subsection{Structure of the paper}\label{ss.structure}

We start by providing rigorous definitions for the cut-off energies~$\tEn_{\alpha,\eps}$ and~$\En_{\alpha,\eps}$ in Sections~\ref{s:tEn} and~\ref{s:En}  (see Definitions~\ref{d:tEn} and~\ref{d:En-eps} below). We then state some of their properties in Section~\ref{s:En-tEn-properties}. In particular, we show that when the energies are large, they are comparable to each other (see Proposition~\ref{prop:EnEquiv}, 
which is analogous to~\eqref{eq:En-tEn} above).

Next, in Section~\ref{s:W} we introduce measures~$\mTn_{\alpha,\eps}[\msp]$ associated with the measure~$\msp$, and describe the effect of an action of ``not too distorting'' diffeomorphisms (see Proposition~\ref{prop:WassEst}).

%
%
%
%

We provide the formal proof of Theorems~\ref{t:main-1}--\ref{t:main-3} in Section~\ref{s:proofs}. The main intermediate result  is Proposition~\ref{t:contr} on  exponential decrease of large energies; it is an analog of (\ref{eq:En-decr}) and is established in Section~\ref{s:exponential}. The final step in the proof of Theorems~\ref{t:main-1}--\ref{t:main-3} is made in Section \ref{s:Holder}.

To make an exposition more reader friendly, the detailed proofs of the most technical statements are postponed till Section \ref{s.tip}. Namely,
in Section \ref{s:properties} we discuss the properties of the functions $\varphi_{\alpha,\eps}(r)$ and $U_{\alpha,\eps}(r)$, and prove Proposition \ref{p:properties}. In Section \ref{s:P41} we provide the proof of Proposition \ref{prop:EnEquiv} that claims that the energies  $\En_{\alpha, \eps}(\msp)$ and $\tEn_{\alpha, \eps}(\msp)$ cannot differ too much.
The proof of Proposition~\ref{prop:WassEst} is given in Section \ref{s:WE}, Proposition \ref{lm:tailsEst} is proven in Section \ref{ss.proof417}, and in Section \ref{s:W-proof} the proof of Lemma \ref{l.West} is given.


Finally, in Appendix \ref{ss.example} we give an example that shows that the finite moment condition (\ref{PosMomentCond}) cannot be replaced by the condition $\int \log \fL(f) \, d\mgr(f) < \infty$ that would be an analog of the condition on distribution in the classical Furstenberg Theorem on random matrix products.

\section{Tools for the proof}\label{s:tools}

In this section, we introduce the $\eps$-cut-off energies $\En_{\alpha,\eps}$ and $\tEn_{\alpha,\eps}$, as well as the corresponding measures $\mT_{\alpha,\eps}(\msp)$ and $\mTn_{\alpha,\eps}(\msp)$ associated to a given measure~$\msp$. We then study their properties and how these objects are changed under an action of a diffeomorphism.

\subsection{$L_2$-type energy $\tEn_{\alpha,\eps}$}\label{s:tEn}

To implement the strategy described in Section~\ref{s.ideas} in a way applicable to any initial probability measure on the manifold, let us start by modifying the definitions of the energy given by (\ref{eq:En-def}) and (\ref{eq:tEn-def}). As $L_2$-point of view is the essential component of the proof, we start by modifying the function used to realise it.

\begin{figure}[!h!]
\includegraphics[height=5cm]{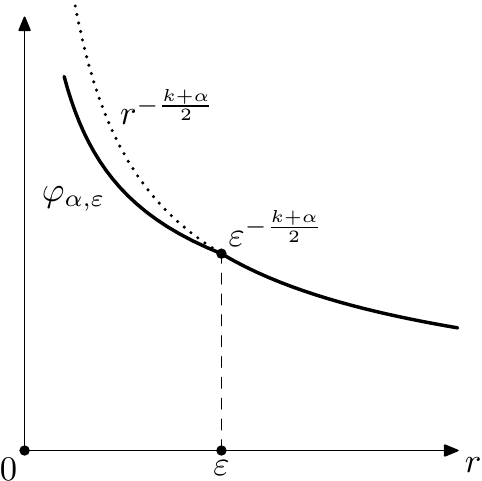}
\caption{Function $\varphi_{\alpha,\eps}(r)$: it coincides with $r^{-\frac{k+\alpha}{2}}$ (graph drawn by dotted line) for $r>\eps$, and scales as a slower-growing power of $r$ for $r<\eps$}\label{f:alpha-eps}
\end{figure}

\begin{definition}\label{d:phi-U}
    We define (see Fig.~\ref{f:alpha-eps})
    \begin{equation}\label{eq:def-varphi}
        \varphi_{\alpha,\eps}(r):=\begin{cases}
            r^{-\frac{k+\alpha}{2}}, & r \ge \eps\\
            \frac{1}{\eps^{\alpha}}\cdot r^{-\frac{k-\alpha}{2}}, & r < \eps.
        \end{cases}
    \end{equation}
\end{definition}

The following definition should be considered  an analog of \eqref{eq:tEn-def} and~\eqref{eq:Theta-m-def} from Section \ref{s.ideas}:

\begin{definition}\label{d:tEn}
    For a measure $\msp \in \mM$, we define a function
    \begin{equation}\label{eq:dens-a-e}
        \dens_{\alpha,\eps}[\msp](y):=\int_M \varphi_{\alpha,\eps}(d(x,y)) d\msp(x).
    \end{equation}
    and a non-probability measure
    \begin{equation*}
        \mT_{\alpha,\eps}[\msp]:=\dens_{\alpha,\eps}^2[\msp](y) \, d\Leb(y).
    \end{equation*}
    We also define
    \begin{equation} \label{eq:tEn}
        \tEn_{\alpha,\eps}(\msp):=\mT_{\alpha,\eps}[\msp](M) = \int_M \dens_{\alpha,\eps}^2[\msp](y) \, d\Leb(y).
    \end{equation}
\end{definition}

In the same way as in Section~\ref{s.ideas}, we rewrite the definition~\eqref{eq:tEn} as a pairwise interaction with some kernel $K_{\alpha,\eps}$:
\begin{lemma}\label{l:triple}
\[
\tEn_{\alpha,\eps}(\msp) = \iint_{M\times M} K_{\alpha,\eps}(x,z) d\msp(x) d\msp(z),
\]
where
\begin{equation}\label{eq:K-a-e}
K_{\alpha,\eps}(x,z) := \int_M \varphi_{\alpha,\eps}(d(x,y)) \varphi_{\alpha,\eps}(d(z,y)) \, d\Leb(y).
\end{equation}
\end{lemma}
\begin{proof}
In the same way as in~\eqref{eq:tEn-kernel}, it suffices to substitute (\ref{d:tEn}) that defines $\dens_{\alpha,\eps}^2[\msp](y)$, into (\ref{eq:tEn}), obtaining a triple integral
\begin{equation}\label{eq:triple}
\tEn_{\alpha,\eps}(\msp) =
\iiint_{M\times M\times M} \varphi_{\alpha,\eps}(d(x,y)) \varphi_{\alpha,\eps}(d(z,y)) \, d\msp(x) d\msp(z) \, d\Leb(y),
\end{equation}
and then change the order of integration.
\end{proof}


\subsection{Interaction-type energy $\En_{\alpha,\eps}$}\label{s:En}

Now let us define  the energy~$\En_{\alpha,\eps}$. As the reader will see, we actually need both of these notions (as well as the relations between them): while $\tEn_{\alpha,\eps}$ is $L_2$-based and thus adapted for the final steps of the proof, it is the energy $\En_{\alpha,\eps}$ that behaves well under an action of a diffeomorphism.

In the same way as in Section~\ref{s.ideas}, we first replace the manifold $M$ with the Euclidean space~$\R^k$ in~\eqref{eq:K-a-e}:
\begin{definition}
    For any two points $x,z\in \R^k$ consider the integral
    \begin{equation}\label{eq:I}
        I_{\alpha,\eps}(x,z):=\int_{\R^k} \varphi_{\alpha,\eps}(|y-x|) \varphi_{\alpha,\eps}(|y-z|) \, d\Leb(y).
    \end{equation}
    Due to the symmetry of the problem, this integral depends only on the distance $r=|x-z|$ between these two points:
    \[
        I_{\alpha,\eps}(x,z)=U_{\alpha,\eps}(|x-z|)
    \]
for some function $U_{\alpha,\eps}$ of $r=|x-z|$. We take this as a definition of~$U_{\alpha,\eps}(r)$. %
\end{definition}


Now we are ready to give a definition of energy $\En_{\alpha,\eps}(\msp)$ that is a modified version of~\eqref{eq:En-def}:
\begin{definition}\label{d:En-eps}
We define
    \[
        \En_{\alpha,\eps}(\msp):=\iint_{M\times M} U_{\alpha,\eps}(d(x,z)) \, d\msp(x) d\msp(z).
    \]
\end{definition}

\begin{figure}[!h!]
\includegraphics[height=5cm]{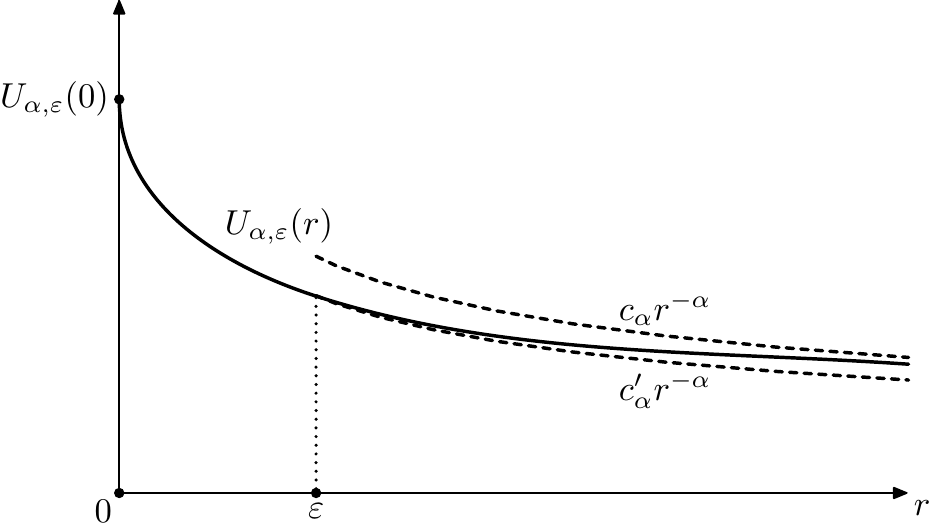}
\caption{The function $U_{\alpha,\eps}(r)$}\label{f:U-a-e}
\end{figure}

Note that the function $U_{\alpha,\eps}(r)$ that we defined can indeed be seen as a cut-off version of $r^{-\alpha}$. Namely, it has the following properties (see~Fig.~\ref{f:U-a-e}). 
\begin{proposition}\label{p:properties}
For any $\alpha\in(0,\frac{1}{2})$ the following hold:
\begin{enumerate}
\item\label{i-1} \textbf{(Finiteness and monotonicity)} For any $\eps>0$ the function $U_{\alpha,\eps}(r)$ is non-increasing and finite everywhere on $[0,+\infty)$;
in particular, for any $r>0$ one has
\begin{equation}\label{eq:upper}
U_{\alpha,\eps}(r)\le U_{\alpha,\eps}(0).
\end{equation}
\item\label{i-3} \textbf{(Outside cut-off: power law)} For some constants $c_{\alpha},c'_{\alpha}$ for any $r\ge \eps>0$ one has
\begin{equation}\label{eq:U-off-eps}
c'_{\alpha} r^{-\alpha} \le U_{\alpha,\eps}(r) \le c_{\alpha} r^{-\alpha};
\end{equation}
also, for any $r>0$
\[
\frac{U_{\alpha,\eps}(r)}{c_{\alpha} r^{-\alpha}} \to 1, \quad \text{ as }  \eps\to 0.
\]
\item\label{i-4} \textbf{($\alpha$-slow changing)} For any $\eps$ and any $r'>r>0$, one has
    \begin{equation}\label{eq:UReg}
\left(\frac{r'}{r}\right)^{-\alpha}\le \frac{U_{\alpha,\eps}(r')}{U_{\alpha,\eps}(r)} \le 1
    \end{equation}
\end{enumerate}
\end{proposition}

We postpone the proof of these properties until Section~\ref{s:properties}. Specifically,
conclusion~\ref{i-1} will be proven in Lemma~\ref{l:U-finite} and Lemma~\ref{lm:fMonotonicity}, conclusion~\ref{i-3} in Corollary~\ref{fHoldCor} and Lemma~\ref{l:lim-U}, and conclusion~\ref{i-4} in~Proposition~\ref{prop:phifReg}.

Meanwhile, we will be using Proposition~\ref{p:properties} to prove our main results. We will assume that for any given $\alpha\in(0,\frac{1}{2})$ the constants $c_\alpha, c'_\alpha$    
are as in the conclusions of this proposition.

Note that the value of $U_{\alpha,\eps}(0)$, appearing in conclusion~(\ref{i-1}) 
of Proposition~\ref{p:properties}, can be calculated explicitly:
\begin{lemma}\label{l:finite}
$U_{\alpha, \eps} (0) = \frac{L}{\alpha}\, \eps^{-\alpha}$, where $L=2\, \Leb B_1(0)$ is a constant.
 \end{lemma}
\begin{proof}
\begin{multline*}
U_{\alpha, \eps} (0) = \int_{\R^k} \varphi_{\alpha,\eps}(|x|)^2 \, d\Leb(x) = \\
        = \int_{|x| < \eps} \left( \frac{|x|^{\frac{\alpha}{2} - \frac{k}{2}}}{\eps^{\alpha}} \right)^2  \, d\Leb(x) + \int_{|x| > \eps} \left( |x|^{-\frac{\alpha}{2} - \frac{k}{2}} \right)^2  \, d\Leb(x) \le \\
        = \Leb(B_1(0)) \left( \frac{r^{\alpha}}{\alpha \eps^{2 \alpha}}\bigg|_{0}^{\eps} - \frac{r^{-\alpha}}{\alpha}\bigg|_{\eps}^{\infty} \right) = \frac{L}{\alpha}\,  \eps^{-\alpha}.
 \end{multline*}
\end{proof}

Due to the monotonicity conclusion~(\ref{i-1}) of Proposition~\ref{p:properties} one gets the following uniform upper bound:
\begin{corollary}\label{c:finite}
For any $\alpha\in(0,\frac{1}{2})$ there exists $C_{\alpha}''$ such that for any measure $\msp$ on $M$ and any $\eps>0$ one has
\[
\En_{\alpha,\eps}(\msp) \le C_{\alpha}'' \eps^{-\alpha}.
\]
\end{corollary}
\begin{proof}
	The function $U_{\alpha, \eps} (r)$ is monotonously decreasing, so
    \begin{equation*}
        \En_{\alpha,\eps}(\msp) = \iint_{M\times M} U_{\alpha,\eps}(d(x,z)) \, d\msp(x) d\msp(z) \le U_{\alpha, \eps} (0) = C_{\alpha}'' \eps^{-\alpha},
    \end{equation*}
    where $C_{\alpha}''=\frac{L}{\alpha}$.
\end{proof}



Proposition~\ref{p:properties} also allows to have H\"older type bounds in terms of the energy, the statement that is an analogue to~(\ref{eq:Markov}) above:
\begin{lemma}\label{l:holdFromEn}
	For any $\alpha, \eps > 0$, any probability measure $\msp$ on $M$ 
one has
	\begin{equation} \label{holdEstFromEn}
	\forall y\in M \quad \forall r>\eps \quad
		\msp(B_{r}(y)) \le  \sqrt{\frac{\En_{\alpha, \eps} (\msp) 2^{\alpha}}{c_{\alpha}'}} \cdot r^{\frac{\alpha}{2}},
	\end{equation}
    where $c_{\alpha}'$ is given by \eqref{eq:U-off-eps} from Proposition~\ref{p:properties}. 
\end{lemma}
\begin{proof}
	One has
	\[
		\En_{\alpha,\eps}(\msp) \ge U_{\alpha,\eps}(2r) \cdot \msp(B_{r}(y))^2 \ge c_{\alpha}' (2r)^{-\alpha}\cdot \msp(B_{r}(y))^2,
	\]
	where the first inequality is the Markov lower bound for the integral (for any $x,z\in B_r(y)$ one has $U_{\alpha, \eps}(d(x,z))\ge U_{\alpha, \eps}(2r)$), and the second is due to~\eqref{eq:U-off-eps}. Dividing by $c_{\alpha}'(2r)^{-\alpha}$ and taking square root, we obtain the desired~(\ref{holdEstFromEn}).
\end{proof}

%
%
%


\subsection{Properties of the energies $\En_{\alpha,\eps}$ and $\tEn_{\alpha,\eps}$}\label{s:En-tEn-properties}

In the same way as for~\eqref{eq:En-tEn}, the two energies, $\En_{\alpha,\eps}$ and $\tEn_{\alpha,\eps}$, are comparable. Namely, we have the following statement.

\begin{proposition}\label{prop:EnEquiv}
    For any $\alpha \in (0, 1/2)$ and any $\delta > 0$ there exists $C>0$ such that for any $\eps>0$ and any measure~$\msp$ on~$M$ such that $\tEn_{\alpha, \eps}(\msp)>C$ or $\En_{\alpha, \eps}(\msp)>C$ one has
    \[
        \frac{\En_{\alpha, \eps}(\msp)}{\tEn_{\alpha, \eps}(\msp)} \in (1-\delta,1+\delta).
    \]
\end{proposition}
The idea of the proof is the same as for~\eqref{eq:En-tEn}: large values of energy can come only from points that are close to each other, and locally any manifold looks like a Euclidean space. Again, we postpone the formal proof of this proposition until Section~\ref{s:P41}.  

It is convenient to re-formulate Proposition~\ref{prop:EnEquiv} using the following notation:
\begin{definition}
We say that two positive numbers, $A$ and $A'$, are \emph{$(\delta,C)$-close}, if
\[
A< (1+\delta)A'+C \quad \text{ and} \quad A'< (1+\delta)A+C;
\]
this can be equivalently rewritten as
\[
\frac{1}{1+\delta} A - \frac{1}{1+\delta}C < A'< (1+\delta)A+C.
\]
We denote it $A \approx_{(\delta,C)} A'$.
\end{definition}

\begin{remark}\label{r:p-4.1}
Proposition~\ref{prop:EnEquiv} can be equivalently reformulated in the following way, that includes all possible probability measures~$\msp$, not necessarily the high energy ones. For any $\alpha \in (0, 1/2)$ and any $\delta > 0$ there exists $C>0$ such that
    \begin{equation}
    	\forall \eps>0 \quad \forall \msp \quad \En_{\alpha,\eps}(\msp)\approx_{(\delta,C)}\tEn_{\alpha,\eps}(\msp).
    \end{equation}
\end{remark}

%
%
%

The next proposition and its corollary give a bound on how the energy $\En_{\alpha,\eps}$ can  change under an action of a diffeomorphism. Namely, it is changed by a factor that is at most $\fL(f)^{\alpha}$, thus once $\alpha$ is close to 0, the factor is close to~$1$.

\begin{proposition} \label{prop:enChangeOneStep}
    For any $\alpha$, any $\eps>0$ and any $f \in \Diff^1(M)$ one has
    \[
        \fL(f)^{-\alpha}\le \frac{\En_{\alpha, \eps}(f_* \msp)}{\En_{\alpha, \eps} (\msp)} \le \fL(f)^{\alpha}.
    \]
\end{proposition}

\begin{proof}
    Indeed, by definition of $\fL(f)$ we know that
    \begin{equation*}
   d(f(x), f(z)) \le \fL(f) d(x, z) \quad \text{for any $x, z \in M$.}
    \end{equation*}
    Proposition~\ref{p:properties} thus implies (see~\eqref{eq:UReg}) that
    \begin{equation*}
    \fL(f)^{-\alpha} U_{\alpha, \eps} (d(x, z))\le     U_{\alpha, \eps} (d(f(x), f(z)))
    \end{equation*}
    Integrating over $d\msp(x) \, d\msp(z)$, one gets the desired
    \[
    \fL(f)^{-\alpha} \En_{\alpha, \eps} (\msp) \le     \En (f_* \msp) .
    \]
    The second inequality follows in the same way from
    \begin{equation*}
   \fL(f)^{-1} d(x, z) \le
   d(f(x), f(z)) 
    \end{equation*}
 \end{proof}

\begin{corollary} \label{cor:EnChangeOneStep}
    For any $\delta>0$ and any $R < \infty$  there exists $\alpha_0>0$ such that for any $\alpha\in(0, \alpha_0)$, any $\eps>0$ and any $f \in \Diff^1(M)$ with $\fL(f) < R$ one has
    \[
        \frac{\En_{\alpha, \eps} (f_*\msp)}{\En_{\alpha, \eps} (\msp)} \in (1-\delta, 1+\delta).
    \]
\end{corollary}


Joining this statement with Proposition~\ref{prop:EnEquiv}, we get the same statement for the energy $\tEn_{\alpha,\eps}$:

\begin{corollary} \label{cor:tEnChangeOneStep}
    For any $\delta>0$ and any $R < \infty$  there exists $\alpha_0>0$ such that for any $\alpha\in(0, \alpha_0)$,
there exists $C>0$ such that for any $f \in \Diff^1(M)$ with $\fL(f) < R$, any $\eps>0$ and any $\msp$ such that $\tEn_{\alpha, \eps}(\msp)>C$ or $\En_{\alpha, \eps}(\msp)>C$ one has
    \begin{equation}\label{eq:tEnChangeOneStep}
        \frac{\tEn_{\alpha, \eps} (f_*\msp)}{\tEn_{\alpha, \eps} (\msp)} \in (1-\delta, 1+\delta).
    \end{equation}
\end{corollary}

A final concluding remark in this direction is that the statement of Corollary~\ref{cor:tEnChangeOneStep} survives if one takes an expectation of energy of a random image, provided that the moment condition~\eqref{PosMomentCond} is satisfied.

\begin{proposition} \label{lm:tailsEst}
    If a probability measure $\mgr\in\mM$ satisfies assumption~\eqref{PosMomentCond} then for any $\delta > 0$ there exists $\alpha_0 > 0$ such that for any $\alpha \in (0,\alpha_0)$ the following holds. There exists $C > 0$ such that for any $\eps > 0$ and any measure $\msp$ on the manifold $M$ with $\tEn_{\alpha, \eps} (\msp) > C$ one has:
    \begin{equation}
        \label{tailsEst1}
        \frac{\mathbb{E}_{\mgr} \left[ \tEn_{\alpha, \eps} (f_* \msp) \right]}{\tEn_{\alpha, \eps} (\msp)} \in (1 - \delta, 1 + \delta).
    \end{equation}
\end{proposition}

For any given individual diffeomorphism $f$ and any measure $\msp$, the quotient $\En_{\alpha, \eps} (f_*\msp)/\En_{\alpha, \eps} (\msp)$ does not exceed $\fL(f)^{\alpha}$ and hence converges to $1$ as $\alpha$ tends to 0 uniformly in~$\eps$. Now, conclusion~\eqref{tailsEst1} is in a sense bringing the expectation under the limit, that can be justified via the Lebesgue dominated convergence theorem. We provide a formal proof of this proposition in Section~\ref{ss.proof417}. 

%

\subsection{Normalizations $\mTn_{\alpha,\eps}$, Wesserstein distance}\label{s:W}

In the same way as in the sketch of the proof, we consider the normalizations $\mTn_{\alpha,\eps}$ of measures $\mT_{\alpha,\eps}$. It is these measures that would turn out to be closer and closer to having a deterministic image assuming that the conclusions of our main results do not hold.
\begin{definition}
\begin{equation}\label{eq:mTn}
	\mTn_{\alpha, \eps} (\msp) = \frac{1}{\tEn_{\alpha, \eps} (\msp)} \, \mT_{\alpha, \eps}[\msp].
\end{equation}
\end{definition}

In what follows it will be convenient to use Wasserstein metric in the space of probability measures on~$M$. Let us recall its definition, as well as definition of the total variation distance between measures:
\begin{definition}
    Let $\msp_1, \msp_2$ be two probability measures on a measure space $(M, \mathcal{B})$. Then the \emph{Wasserstein distance} between them is defined as
    \begin{equation*}
        W(\msp_1, \msp_2) = \inf_{\gamma} \iint_{M\times M} d(x, y) \, d \gamma(x, y),
    \end{equation*}
    where the infimum is taken over all probability measures $\gamma$ on $(M \times M, \mathcal{B} \times \mathcal{B})$ with the marginals (projections on the $x$ and $y$ coordinates) $P_x(\gamma) = \msp_1$ and $P_y(\gamma) = \msp_2$.
\end{definition}
\begin{definition}
    Let $\msp_1, \msp_2$ be two probability measures on a measure space $(M, \mathcal{B})$. Then the \emph{total variation} distance between them is
    \[
    	\TV(\msp_1, \msp_2) = \sup_{B \in \mathcal{B}} |\msp_1(B) - \msp_2(B)|.
    \]
\end{definition}

Also, we will need the following statement that can be found, for example, in \cite[Theorem 6.15]{V}: the Wasserstein metric is bounded from above by a the diameter times the total variation distance.
\begin{lemma} \label{lm:WasEst}
    For any two probability measures $\msp_1, \msp_2 $ on a manifold $M$ we have
    \begin{equation*}
        W(\msp_1, \msp_2) \le \diam(M) \cdot \TV(\msp_1, \msp_2).
    \end{equation*}
\end{lemma}

We are going to use the following statement. Assuming that $\alpha$ is sufficiently small, the energy of a measure $\msp$ is sufficiently high, and a diffeomorphism $f$ is ``not too distorting'', not only are the energies of $\msp$ and of~$f_*\msp$  close to each other, but also the measure $\mTn_{\alpha,\eps}[f_*\msp]$ is close to the push-forward $f_* \mTn_{\alpha,\eps}[\msp]$.

\begin{proposition} \label{prop:WassEst}
	For any $\delta > 0$ and any $R > 0$ there exists $\alpha_0 > 0$ such that for any $\alpha \in (0,\alpha_0)$ there exists $C > 0$ such that for any $f \in \Diff^1(M)$ with $\fL(f) < R$, any $\eps > 0$, and any $\msp$ such that $\tEn_{\alpha, \eps} (\msp) > C$ or $\En_{\alpha, \eps} (\msp) > C$  one has
	\begin{equation} \label{tEnWassInv}
		W(f_* \mTn_{\alpha, \eps} (\msp), \mTn_{\alpha, \eps} (f_* \msp)) < \delta.
	\end{equation}
\end{proposition}

The idea of the proof of this proposition is close to the one for Proposition~\ref{prop:EnEquiv}: large energy means that most of it comes from points close to each other. Moreover, considering the energy as a triple integral~\eqref{eq:triple}, one sees that most of it comes out from all the three points being close to each other.
Now, when a diffeomorphism $f$ is applied, these local contributions are changed \emph{roughly} in the same way as the kernel~$U_{\alpha,\eps}(x,z)$. Finally, the latter almost does not change if $\fL(f)^{\alpha}$ is close to~1, thus implying the desired closeness of two measures.

Once again, we postpone the formal proof of this proposition until Section~\ref{s:WE}. 


\section{Proofs of main results}\label{s:proofs}

In this section we prove Theorems \ref{t:main-1}, \ref{t:main-2}, and \ref{t:main-3}, except for the technically complicated parts, that we moved to Section~\ref{s.tip} (specifically,  the proofs of Propositions~\ref{p:properties},~\ref{prop:EnEquiv},~\ref{lm:tailsEst},~\ref{prop:WassEst}, together with the proof of Lemma~\ref{l.West}).


\subsection{Exponential decrease of large energies}\label{s:exponential}

The key step of the proof of our main results is the following proposition:

\begin{proposition} \label{t:contr}
	If the measure $\mgr$ satisfies \textbf{finite moment condition} and \textbf{no deterministic images} then there exist $\alpha > 0, C < \infty, \lambda < 1$, such that for any $\eps > 0$ if $\tEn_{\alpha, \eps} (\msp) > C$ then
	\begin{equation} \label{EnContr}
		\tEn_{\alpha, \eps} (\mgr * \msp) < \lambda \tEn_{\alpha, \eps} (\msp).
	\end{equation}
\end{proposition}

\begin{proof}
   Fix some small number $\delta > 0$, take $\lambda=1-\delta$, and assume the contrary: for any $\alpha > 0$ and $C < \infty$ there exists an $\eps > 0$ and a probability measure $\msp$ with energy $\tEn_{\alpha,\eps}(\msp)>C$, such that
	\begin{equation} \label{contrAssumption2}
		\tEn_{\alpha, \eps} (\mgr * \msp) \ge (1-\delta) \tEn_{\alpha, \eps} (\msp).
	\end{equation}

Recall the following standard statement:
\begin{lemma}\label{l:var}
Assume that in some Hilbert space $\mathcal{H}$ a probability measure is given; in other words, one is given a random variable $v$ taking values in this space. Assume that this measure has finite second moment, and let $\bv:=\E v$ be its expectation. Then
\begin{equation}\label{eq:v}
\E \langle v-\bv,v-\bv\rangle = \E \langle v,v\rangle - \langle \bv,\bv\rangle
\end{equation}
\end{lemma}
In dimension one, it is exactly the equivalence between the two definitions of variance of a random variable; and in general, it is  straightforward to check (\ref{eq:v}): it suffices to substitute $v=(v-\bv)+\bv$ to the expectation in the right hand side and use the linearity of the scalar product.

Now, let us apply this statement to~$L_2(M, \text{\rm Leb})$. Namely, for a given measure $\msp$ its averaged image $\mgr*\msp$ is the expectation of $f_*\msp$, where the diffeomorphism $f$ of $M$ is taken randomly w.r.t. the measure~$\mgr$.
Hence, the same applies for the density~$\dens_{\alpha,\eps}[\msp]$, given by~\eqref{eq:dens-a-e} (see Definition~\ref{d:tEn}):
\begin{equation}\label{eq:rho-expect}
\dens_{\alpha,\eps}[\mgr*\msp] = \E_{\mgr}  \, \dens_{\alpha,\eps}[f_* \msp].
\end{equation}
Substituting this into~\eqref{eq:v} (with $v=\dens_{\alpha,\eps}[f_* \msp]$), and taking into account the definition~\eqref{eq:tEn}, we get
\begin{multline} \label{ContrThmVarFormula}
		\mathbb{E}_{\mgr} \left[ \int_M \left( \dens_{\alpha, \eps}[f_* \msp] (y) - \dens_{\alpha, \eps}[\mgr * \msp] (y) \right)^2 d \Leb(y) \right]=
		\\
		= \mathbb{E}_{\mgr} \left[ \tEn_{\alpha, \eps} (f_* \msp) \right] - \tEn_{\alpha, \eps} (\mgr * \msp);
	\end{multline}
in particular,
\begin{equation}\label{eq:avg-bound}
 \tEn_{\alpha, \eps} (\mgr * \msp) \le \mathbb{E}_{\mgr} \left[ \tEn_{\alpha, \eps} (f_* \msp) \right]
\end{equation}

By Proposition~\ref{lm:tailsEst} we can find sufficiently small $\alpha  > 0$ and sufficiently large $C < \infty$, such that for any $\eps>0$ and any measure $\msp$ on $M$
with $\tEn_{\alpha, \eps}(\msp) > C$ the inequality~\eqref{tailsEst1} holds, and thus (taking only the upper bound)
\begin{equation}\label{eq:E-f-delta}
        \mathbb{E}_{\mgr} \left[ \tEn_{\alpha, \eps} (f_* \msp) \right]
         \le (1 + \delta)\tEn_{\alpha, \eps} (\msp).
\end{equation}

Denote by $\msp$ a probability measure with $\tEn_{\alpha, \eps}(\msp) > C$, for which the inequality~\eqref{contrAssumption2} holds. Substituting~\eqref{eq:E-f-delta} and~\eqref{contrAssumption2} in the right hand side of~\eqref{ContrThmVarFormula}, we get

	\begin{equation} \label{eq:exp-rho}
		\mathbb{E}_{\mgr} \left[ \int_M \left( \dens_{\alpha, \eps} (f_* \msp) (y) - \dens_{\alpha, \eps} (\mgr * \msp) (y) \right)^2 d \Leb(y) \right] < 2 \delta \, \tEn_{\alpha, \eps} (\msp).
	\end{equation}
and thus
	\begin{equation} \label{ExpEst1}
		\mathbb{E}_{\mgr} \left[ \int_M \frac{\left( \dens_{\alpha, \eps} (f_* \msp) (y) - \dens_{\alpha, \eps} (\mgr * \msp) (y) \right)^2 }{\tEn_{\alpha, \eps} (\msp)} d \Leb(y) \right] < 2 \delta .
	\end{equation}

    The final step is to show that measures
    $$
    \mTn_{\alpha, \eps} (f_* \msp)=\dfrac{\dens_{\alpha, \eps} (f_* \msp) (y)^2}{\tEn_{\alpha, \eps} (f_* \msp)} d \Leb(y) \ \  \text{and} \ \ \mTn_{\alpha, \eps}(\mgr * \msp)=\dfrac{\dens_{\alpha, \eps} (\mgr * \msp) (y)^2}{\tEn_{\alpha, \eps} (\mgr * \msp)} d \Leb(y)
    $$
     are close with high probability.
     And indeed, we have the following statement, estimating (even) the total variations distance:
    \begin{lemma}\label{l.West}
    Under the assumptions above
\begin{equation}\label{eq:Probably-TV}
    \p_{\mgr} \left[ \TV (\mTn_{\alpha, \eps} (f_* \msp), \mTn_{\alpha, \eps}(\mgr * \msp))
    		> 10 \sqrt[8]{\delta} \right]
 	   < 4 \sqrt{\delta},
\end{equation}
and hence
\begin{equation}\label{eq:Probably}
    \p_{\mgr} \left[ W(\mTn_{\alpha, \eps} (f_* \msp), \mTn_{\alpha, \eps}(\mgr * \msp)) > 10 \diam(M) \sqrt[8]{\delta} \right] < 4 \sqrt{\delta}.
\end{equation}
    \end{lemma}
    Proof of Lemma~\ref{l.West} is the most technical part of the proof of Proposition~\ref{t:contr},
    and we postpone it till Section~\ref{s:W-proof}. 

\vspace{3pt}

We are now ready to conclude the proof of Proposition~\ref{t:contr}. Namely, the conclusion~\eqref{eq:Probably} states that with high probability  the measure $\mTn_{\alpha,\eps}(f_* \msp)$ is close to the deterministic one, $\mTn_{\alpha,\eps}(\mgr * \msp)$. The next step is to show that with high probability the first measure is close to $f_*$-image of a given measure,~$\mTn_{\alpha,\eps}(\msp)$.

To do so, choose $R < \infty$, such that
    \begin{equation*}
        \p_{\mgr} \left[ \fL(f) < R \right] > 1 - \sqrt[4]{\delta}
    \end{equation*}
    By Proposition \ref{prop:WassEst} we can choose $C$ big enough, so that for any $f \in \Diff^1(M)$, such that $\fL(f) < R$ inequality (\ref{tEnWassInv}) holds, and using triangle inequality we arrive to
	\begin{equation} \label{limitLaw}
		\p_{\mgr} \left[ W(f_* \mTn_{\alpha, \eps} (\msp), \mTn_{\alpha, \eps}(\mgr * \msp)) > 10 \diam(M) \sqrt[8]{\delta} + \delta \right] < 5 \sqrt[4]{\delta}.
	\end{equation}

	Now consider the sequence $\delta_n = \frac{1}{n}$ and denote by $\msp_n$ the corresponding measure and corresponding parameters by $\alpha_n, \eps_n$, for which (\ref{contrAssumption2}) holds. Then the measures $\mTn_{\alpha_n, \eps_n} (\msp_n)$ have a weakly convergent subsequence. Futhermore, we can take another subsequence, such that the measures $\mTn_{\alpha_{n_k}, \eps_{n_k}} (\mgr * \msp_{n_k})$ also converge. It remains to notice that due to inequality (\ref{limitLaw}) the limit measure
	\begin{equation*}
		m = \lim_{k \to \infty} \mTn_{\alpha_{n_k}, \eps_{n_k}}(\msp_{n_k})
	\end{equation*}
	has an almost surely constant image under the action of $f \in \supp(\mgr)$, that is equal to
	\begin{equation*}
		\tilde{m} = \lim_{k \to \infty} \mTn_{\alpha_{n_k}, \eps_{n_k}} (\mgr * \msp_{n_k}),
	\end{equation*}
	which  contradicts our assumption.	
\end{proof}

Finally, Proposition~\ref{t:contr} can be generalized by the following statement.

\begin{proposition} \label{rem:uniformContr}
    Let $\mK\subset \mM$ be a compact subset that satisfies the following conditions:
    \begin{itemize}
        \item \textbf{(uniform finite moment condition)} There exists $C_0, \gamma>0$ such that for any $\mgr\in\mK$ one has
        \begin{equation*}
            \int \fL(f)^\gamma \, d \mgr(f) < C_0;
        \end{equation*}
        \item \textbf{(no deterministic images)} For any $\mgr\in\mK$ there are no probability measures $\msp,\msp'$ on~$M$ such that $f_* \msp = \msp'$ for $\mgr$-almost all $f\in \Diff^1(M)$.
    \end{itemize}
    Then the constants $\alpha, \lambda$ and $C$ in Proposition~\ref{t:contr} can be chosen uniformly in $\mgr \in \mK$.
\end{proposition}
\begin{proof}
    Indeed, assume the contrary. It is easy to see that in the statements in~Section~\ref{s:tools} the constants can be chosen uniformly for all~$\mgr\in\mK$. Now, repeating the proof of Proposition~\ref{t:contr}, we see that the only part to be modified is passage to the limit. Again taking $\delta_n=\frac{1}{n}$ one finds $\alpha_n,\eps_n$ and a measure $\mgr_n\in \mK$ such that~\eqref{contrAssumption2} holds.
Passing to a subsequence $(n_k)$, one can ensure the existence of three limits:
    	\begin{equation*}
		m = \lim_{k \to \infty} \mTn_{\alpha_{n_k}, \eps_{n_k}}(\msp_{n_k})
	\end{equation*}
	\begin{equation*}
		\tilde{m} = \lim_{k \to \infty} \mTn_{\alpha_{n_k}, \eps_{n_k}} (\mgr_{n_k} * \msp_{n_k}),
	\end{equation*}
	\begin{equation*}
		\tilde{\mgr} = \lim_{k \to \infty} \mgr_{n_k}
	\end{equation*}
	In turn, this implies that for $\tilde{\mgr}$-a.e. diffeomorphism $f$ one has $f_* m = \tilde{m}$, thus obtaining a contradiction with the ``no deterministic images'' assumption at $\tilde{\mgr}\in\mK$.
\end{proof}

Finally, we note that the conclusion of Proposition~\ref{t:contr} can be modified to include all possible measures~$\msp$ (without assuming the energy higher than~$C$). To do so, we need to adjust the upper bound~\eqref{EnContr} in order to include low-energy measures~$\msp$:


\begin{lemma}\label{l:uniform-C}
Assume that the measure $\mgr$ satisfies the finite moment condition (\ref{PosMomentCond}) with some $\gamma$, and that $\alpha<\gamma$ and $C$ are given. Then there exists $C'$ such that for any measure $\msp$ on $M$ with
$\En_{\alpha, \eps} (\msp) \le C$ one has
\begin{equation}
\En_{\alpha, \eps} (\mgr*\msp) \le C'.
\end{equation}
\end{lemma}
\begin{corollary}\label{c:En}
In the assumptions of Proposition~\ref{t:contr} one can conclude that there exist
$\alpha > 0, \tC < \infty, \lambda < 1$, such that for any $\eps > 0$ and any measure $\msp$ on $M$
	\begin{equation} \label{eq:EnContr}
		\En_{\alpha, \eps} (\mgr * \msp) < \max(\lambda \En_{\alpha, \eps} (\msp),\tC)
	\end{equation}
\end{corollary}

\begin{proof}[Proof of Lemma~\ref{l:uniform-C}]
Note first that due to Proposition~\ref{prop:EnEquiv}, for any $\alpha$ and $C_1$ there exists $C_2$ such that for any $\eps$ and any measure $\msp$ one has
\[
\tEn_{\alpha,\eps}(\msp)<C_1 \Rightarrow \En_{\alpha,\eps}(\msp)<C_2
\]
and vice versa,
\begin{equation}\label{eq:t-bound}
\En_{\alpha,\eps}(\msp)<C_1 \Rightarrow \tEn_{\alpha,\eps}(\msp)<C_2.
\end{equation}
In other words, having an upper bound on one of the energies $\En_{\alpha,\eps}(\msp)$, $\tEn_{\alpha,\eps}(\msp)$ implies a bound for the other one.

Next, the finite moment condition implies the finiteness of the expectation
\begin{equation}\label{eq:I-exp}
C_I:= \E_{\mgr}\fL(f)^{\alpha} < \infty.
\end{equation}
Now,~\eqref{ContrThmVarFormula} implies that for any measure $\msp$ we have
\[
\tEn_{\alpha,\eps}(\mgr*\msp) \le \E_{\mgr} \tEn_{\alpha,\eps}(f_* \msp);
\]
At the same time
Proposition~\ref{prop:EnEquiv} for $\delta=\frac{1}{2}$
 implies that for some constant $C_1$ one has for any measure $\msp'$
\[
\tEn_{\alpha,\eps}(\msp') \le \max(2\En_{\alpha,\eps}(\msp'), C_1 ).
\]
Applying this for $\msp'=f_* \msp$ and joining it with Proposition~\ref{prop:enChangeOneStep}, we get
\[
\tEn_{\alpha,\eps}(f_*\msp) \le \max(2\En_{\alpha,\eps}(f_*\msp), C_1 ) \le 2\fL(f)^{\alpha} \En_{\alpha,\eps}(\msp)+ C_1.
\]
Taking the expectation w.r.t.~$\mgr$ and using~\eqref{eq:I-exp}, we finally get a uniform bound
\[
\tEn_{\alpha,\eps}(\mgr*\msp) \le 2C_I C+C_1 =: C'.
\]
\end{proof}

\begin{proof}[Proof of Corollary~\ref{c:En}]
In the proof of Proposition~\ref{t:contr} we can take $\alpha$ arbitrarily small, so we can assume that $\alpha<\gamma$.
Applying Proposition~\ref{prop:EnEquiv} to both sides of~\eqref{EnContr}, where we take $\delta$ sufficiently small so that $\frac{1-\delta}{1+\delta}>\lambda$, allows to conclude that there exists some constant $C_3$ such that for all $\eps$ and $\msp$
\[
\En_{\alpha, \eps} (\mgr * \msp) < \lambda' \En_{\alpha, \eps} ( \msp) \quad \text{ if } \, \En_{\alpha, \eps} ( \msp) > C_3,
\]
where $\lambda'=\frac{1+\delta}{1-\delta} \lambda <1$.
Meanwhile, Lemma~\ref{l:uniform-C} implies that there exists some $C_4$ such that
\[
\En_{\alpha, \eps} (\mgr * \msp) <C_4 \quad \text{ if } \, \En_{\alpha, \eps} ( \msp) \le C_3.
\]
Taking $\tC:=\max(C_3,C_4)$, we get the desired~\eqref{eq:EnContr}.
\end{proof}

\subsection{H\"older bounds: proofs of Theorems~\ref{t:main-1}--\ref{t:main-3}}\label{s:Holder}

  Now we are ready to complete the proof of Theorem~\ref{t:main-2} (and thus of Theorem~\ref{t:main-1}).

\begin{proof}[Proof of Theorem~\ref{t:main-2}]

Let $\alpha,\tC$ be as in Corollary~\ref{c:En}: for any $\eps$ and for any measure $\msp$ one has
\[
		\En_{\alpha, \eps} (\mgr * \msp) < \max(\lambda \En_{\alpha, \eps} (\msp),\tC);
\]
applying this $n$ times, we get
\[
		\En_{\alpha, \eps} (\mgr^n * \msp) < \max (\lambda^n \En_{\alpha, \eps} (\msp),\tC).
\]
Recall that Lemma~\ref{l:finite} gives a uniform upper bound $\En_{\alpha, \eps} (\msp)\le C_{\alpha}'' \eps^{-\alpha}$, where $C_{\alpha}''=\frac{L}{\alpha}$; choose $\eps:=\lambda^{\frac{n}{\alpha}}$, then we have
	\begin{equation*}
		\En_{\alpha, \eps} (\mgr^n * \msp) < \max \left( 2C, C_{\alpha}'' \right).
	\end{equation*}
Now, applying Lemma~\ref{l:holdFromEn} for any $r > \eps = \kappa^n$, where $\kappa:=\lambda^{\frac{1}{\alpha}}$, we deduce that
	\begin{equation*}
		(\mgr^n * \msp) (B_r(x)) < \sqrt{\frac{2^{\alpha} \max \left( 2C, C_{\alpha}'' \right)}{c_{\alpha}'}} \cdot r^{\frac{\alpha}{2}},
	\end{equation*}
	which completes the proof of Theorem \ref{t:main-2}.

\end{proof}

    Arguing in the same way as in the proof of Theorem~\ref{t:main-2}, 
    Theorem~\ref{t:main-3} can be deduced from Proposition~\ref{rem:uniformContr}.

\begin{remark}
    The same proof works for the setting mentioned in Remark \ref{rem:boundary}. One has to consider an embedding $M \hookrightarrow \tilde{M}$ into a closed manifold $\tilde{M}$, to define energies $\En_{\alpha, \eps}$ and $\tEn_{\alpha, \eps}$ with respect to the $\Leb$ measure on $\tilde{M}$ and to repeat 
    the original proof verbatim.
\end{remark}


\section{Proofs of used properties}\label{s.tip}

Here we provide the proofs of the technical statements that were formulated and used above, as described in Section \ref{ss.structure}.

\subsection{Properties of $\varphi_{\alpha,\eps}(r)$ and $U_{\alpha,\eps}(r)$}\label{s:properties}

Here we prove Proposition \ref{p:properties}.
Let us start with some scaling properties of~$\varphi_{\alpha,\eps}$ and of~$U_{\alpha,\eps}$:

%
\begin{lemma}
    For any $\lambda >0$ and any $\alpha,\eps>0$ we have
    \begin{equation} \label{eq:phi-scaling}
        \forall r>0 \quad \varphi_{\alpha,\lambda\eps}(\lambda r)=\lambda^{-\frac{k}{2}-\frac{\alpha}{2}} \varphi_{\alpha, \eps}(r)
    \end{equation}
    and
    \begin{equation} \label{eq:f-scaling}
        \forall r>0 \quad U_{\alpha,\lambda \eps}(\lambda r)=\lambda^{-\alpha} U_{\alpha,\eps}(r).
    \end{equation}
    In particular, one has
    \begin{equation} \label{eq:scalings}
        U_{\alpha,\eps}(r) = r^{-\alpha} U_{\alpha,\frac{\eps}{r}}(1) = \eps^{-\alpha} U_{\alpha,1}(r/\eps).
    \end{equation}
\end{lemma}

\begin{proof}
Relation~\eqref{eq:phi-scaling} follows directly from~\eqref{eq:def-varphi} in the definition of~$\varphi_{\alpha,\eps}$.
    For~\eqref{eq:f-scaling}, it suffices to make a change of variable~$y=\lambda y'$ in the integral~\eqref{eq:I}. Finally,~\eqref{eq:scalings} immediately follows from~\eqref{eq:f-scaling}
\end{proof}

\begin{corollary} \label{fHoldCor}
    For any $\alpha>0$ and any $r>\eps>0$ one has
    \[
        U_{\alpha,\eps}(r) \ge c_{\alpha}' r^{-\alpha},
    \]
    where $c_{\alpha}':=U_{\alpha,1}(1)$.
\end{corollary}

\begin{proof}
    As a function of $\eps$, the function $\varphi_{\alpha,\eps}(r)$ and hence $U_{\alpha,\eps}(r)$ is decreasing. The conclusion thus directly follows from~\eqref{eq:scalings}.
\end{proof}

\begin{proposition} \label{prop:phifReg}
    \begin{equation} \label{phiReg}
        \forall \alpha,\eps > 0 \quad \forall r' > r > 0 \quad (r' / r)^{-\frac{k}{2} - \frac{\alpha}{2}} \le \frac{\varphi_{\alpha,\eps}(r')}{\varphi_{\alpha,\eps}(r)} \le (r' / r)^{-\frac{k}{2} + \frac{\alpha}{2}}.
    \end{equation}
    \begin{equation}\label{UReg}
        \forall \alpha,\eps>0 \quad \forall r'>r>0 \quad (r'/r)^{-\alpha}\le \frac{U_{\alpha,\eps}(r')}{U_{\alpha,\eps}(r)} \le (r'/r)^{\alpha}.
    \end{equation}
\end{proposition}

\begin{proof}
Note first that, on each of the intervals $(0,\eps]$ and $[\eps,\infty)$ the function $\varphi_{\alpha,\eps}$ takes the form $\const \cdot r^{-\frac{k}{2}} r^{\pm \frac{\alpha}{2}}$, thus implying the first statement on each of these intervals. Now, to handle the case $r'>\eps>r$, it suffices to multiply the obtained estimates for $(r,\eps)$ and for $(\eps,r')$.
%

%
%


To establish the second conclusion, set $\lambda:=\frac{r}{r'}$, $\lambda<1$.
Take two points $x',z'$ at the distance $r'$ from each other, then the points $x=\lambda x', \, z =\lambda z'$ are at distance $r$. Let us make a change of variable $y=\lambda y'$ in~\eqref{eq:I} for $I_{\alpha,\eps}(x,z)$:
\begin{multline*}
U_{\alpha,\eps}(r)=I_{\alpha,\eps}(x,z) =
\\
= \int_{\R^k} \varphi_{\alpha,\eps}(|y-x|) \varphi_{\alpha,\eps}(|y-z|) \, d\Leb(y) =
\\
= \int_{\R^k} \varphi_{\alpha,\eps}(|\lambda y'-\lambda x'|) \varphi_{\alpha,\eps}(|\lambda y'-\lambda z'|) \, d\Leb(\lambda y')
\\ = \lambda^k \int_{\R^k} \varphi_{\alpha,\eps}(\lambda |y'-x'|) \varphi_{\alpha,\eps}(\lambda |y'-z'|) \, d\Leb(y')
\end{multline*}
Now, due to~\eqref{phiReg},
\[
\lambda^{-(k-\alpha)}\le \frac{\varphi_{\alpha,\eps}(\lambda |y'-x'|) \varphi_{\alpha,\eps}(\lambda |y'-z'|)}{\varphi_{\alpha,\eps}(|y'-x'|) \varphi_{\alpha,\eps}(|y'-z'|)} \le \lambda^{-(k+\alpha)}
\]
Hence, $U_{\alpha,\eps}(r)$ differs from
\[
\lambda^k \int_{\R^k} \varphi_{\alpha,\eps}(|y'-x'|) \varphi_{\alpha,\eps}(|y'-z'|) \, d\Leb(y') = \lambda^k U(r')
\]
by a factor that is between $\lambda^{-(k-\alpha)}$ and $\lambda^{-(k+\alpha)}$. As $\lambda^k$ cancels out, we get the desired
\[
\lambda^{\alpha} U(r') \le U(r) \le \lambda^{-\alpha} U(r').
\]

%
%
%
%
%
%
\end{proof}


A first immediate remark is the finiteness of $U$:
\begin{lemma}\label{l:U-finite}
For any $\alpha\in(0,\frac{1}{2})$, $\eps>0$, and any $r\ge 0$ the value $U_{\alpha,\eps}(r)$ is finite, and
\begin{equation}
\forall r\ge 0 \quad U_{\alpha,\eps}(r) \le \langle \varphi_{\alpha,\eps},\varphi_{\alpha,\eps}\rangle_{L_2(\R^k)} = U_{\alpha,\eps}(0).
\end{equation}
\end{lemma}
\begin{proof}
Note that the $\varphi_{\alpha,\eps}\in L_2(\R^k)$. Indeed, the inequality $2\cdot \frac{k-\alpha}{2}<k$ ensures squared integrability at the origin, and as $2\cdot \frac{k+\alpha}{2}>k$, the integral converges at infinity, too.

Now, $U(r)$ is a scalar product between two $r$-shifted copies of $\varphi_{\alpha,\eps}$. Hence, it is finite for all $r\ge 0$, and we get the desired~\eqref{eq:upper}.
\end{proof}

The next lemma explains that the interaction potential $U_{\alpha,\eps}(r)$ can be considered as a cut-off of $r^{-\alpha}$ (or, more precisely, of $c_{\alpha} r^{-\alpha}$):
\begin{lemma}\label{l:lim-U}
For any $\alpha\in (0,\frac{1}{2})$ the function $U_{\alpha,\eps}(r)$ is decreasing in~$\eps$, and one has
\begin{equation}\label{eq:lim-U}
\forall r>0 \quad \lim_{\eps\to 0} U_{\alpha,\eps}(r) = c_{\alpha} r^{-\alpha} 
\end{equation}
where
\[
c_{\alpha}:= \int_{\R^k} |y|^{-\frac{k+\alpha}{2}} |y-(1,0,\dots,0)|^{-\frac{k+\alpha}{2}} \, d\Leb(y)
\]
is a constant depending only on $\alpha$.
\end{lemma}
\begin{proof}
The monotonicity part immediately comes from the monotonicity of $\varphi_{\alpha,\eps}$ in $\eps$. As the functions under the integral in~\eqref{eq:I} are positive and monotonous in $\eps$, one can pass to the limit as $\eps\to 0$, obtaining the convolution square of~$r^{-\frac{k+\alpha}{2}}$ in $\R^k$. The latter equals $c_\alpha r^{-\alpha}$ due to the scaling reasons, and the constant $c_\alpha$ comes from substituting $r=1$.
\end{proof}

\begin{lemma} \label{lm:fMonotonicity}
    For any $\alpha, \eps > 0$ function $U_{\alpha, \eps} (r)$ is non-increasing in~$r$.
\end{lemma}

\begin{proof}
Actually, the convolution of any positive, spherically symmetric, radially non-increasing and square integrable functions $\psi_1(r)$, $\psi_2(r)$ in $\R^k$ is again spherically symmetric radially non-increasing function. Indeed, any such function can be approximated from below by a linear combination of indicator functions $\Ind_{B_R(0)}(x)=\Ind_{|x|<R}$ of balls centred at the origin, taken with positive coefficients. It now suffices to establish the statement for such indicator functions, as one can then pass to the (monotonous) limit.

Therefore it is enough to show that the volume of the intersection of two such balls $B_{R_1}(0)$ and $B_{R_2}(x)$ is a non-increasing function of~$|x|$ (see Fig.~\ref{f:B-r}). But that is a partial case of the following simple statement:

\begin{figure}[!h!]
\includegraphics[height=3.5cm]{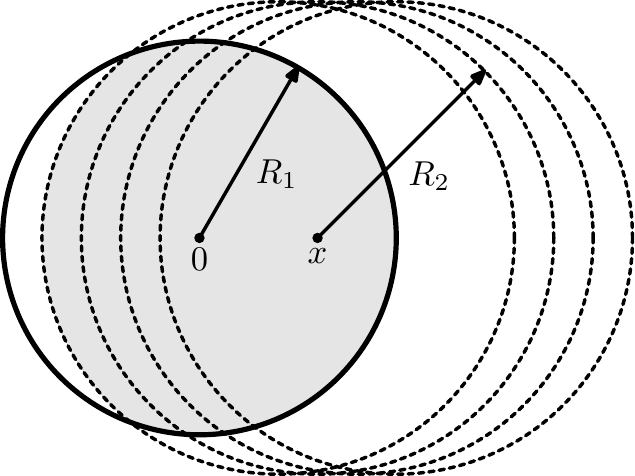}
\caption{Intersections of $B_{R_1}(0)$ with $B_{R_2}(x)$}\label{f:B-r}
\end{figure}

\begin{lemma}\label{l.conv}
Suppose $A\subset \R^k$ is convex, $B\subset A$ is measurable, and $\text{\rm Leb}\,(B)<\infty$. Then for any non-zero vector $\bar v\in \R^k$ the measure  $\text{\rm Leb}\,(A\cap (B+t\bar v))$ as a function of $t\ge 0$ is non-increasing.
\end{lemma}

Indeed, monotonicity in Lemma \ref{l.conv} follows from the fact that due to convexity of $A$, for any point $b\in B$, as soon as $b+t_1\bar v\not \in A$, we must have $b+t\bar v\not \in A$ for all $t>t_1$.

%

\end{proof}

\subsection{Comparing $\En_{\alpha, \eps}(\msp)$ and $\tEn_{\alpha, \eps}(\msp)$: Proof of Proposition \ref{prop:EnEquiv}}\label{s:P41}

We will need the following statement:
\begin{lemma}\label{l:K-U-comp}
For any $\alpha>0$, the interaction potentials $K_{\alpha,\eps}(x,z)$ and $U_{\alpha,\eps}(d(x,z))$ are comparable in the  following sense:
\begin{multline}\label{eq:K-U-close}
\forall \delta>0 \, \exists C: \quad \forall \eps>0, \,\forall x,z\in M \\
K_{\alpha,\eps}(x,z) \approx_{(\delta, C)} U_{\alpha,\eps}(d(x,z)), \ \text{i.e. }\\
\frac{1}{1+\delta} (U_{\alpha,\eps}(d(x,z))-C) < K_{\alpha,\eps}(x,z) < (1+\delta) U_{\alpha,\eps}(d(x,z))+C.
\end{multline}
\end{lemma}

Postponing its proof for the moment, note that it implies Proposition~\ref{prop:EnEquiv}.

\begin{proof}[Proof of Proposition~\ref{prop:EnEquiv}]
It suffices to integrate~\eqref{eq:K-U-close} w.r.t. $\msp\times \msp \, (x,z)$. As constant $C$ does not depend on these points nor on the measure $\msp$, one gets
\[
\frac{1}{1+\delta} (\En(\msp)-C) < \tEn(\msp) < (1+\delta) \En(\msp)+C
\]
with the same constant~$C$. As it was noticed in Remark~\ref{r:p-4.1}, this is an equivalent form of Proposition~\ref{prop:EnEquiv}.
\end{proof}

\begin{proof}[Proof of Lemma~\ref{l:K-U-comp}]
Let $\alpha>0$ be given. Take any $r_0>0$ and divide the integral~\eqref{eq:K-a-e} defining $K_{\alpha,\eps}(x,z)$ into two parts, depending on whether the distance $d(x,y)$ exceeds $r_0$:
\begin{multline*}
K_{\alpha,\eps}(x,z) = \int_{B_{r_0}(x)} \varphi_{\alpha,\eps}(d(x,y)) \varphi_{\alpha,\eps}(d(z,y)) \, d\Leb(y) +
\\
+ \int_{M\setminus B_{r_0}(x)} \varphi_{\alpha,\eps}(d(x,y)) \varphi_{\alpha,\eps}(d(z,y)) \, d\Leb(y).
\end{multline*}
Denote the first and the second summands as $K_{\alpha,\eps}^{(r_0)}(x,z)$ and $\bK_{\alpha,\eps}^{(r_0)}(x,z)$.

Note that the second summand is uniformly bounded: indeed, the factor $\varphi_{\alpha,\eps}(d(x,y))$ doesn't exceed a constant $\varphi_{\alpha, \eps}(r_0)\le \varphi_{\alpha}(r_0)$, while the second factor $\varphi_{\alpha,\eps}(d(y,z))$ is a function with the integral on~$M$ that is bounded uniformly in~$z\in M$. The latter uniform bound can be seen by again decomposing the integral in two:
\begin{multline}\label{eq:phi-integral-bound}
\int_M \varphi_{\alpha,\eps}(d(y,z)) \, d\Leb(y) = \int_{B_{r_0}(z)} \varphi_{\alpha,\eps}(d(y,z)) \, d\Leb(y) +
\\ + \int_{M\setminus B_{r_0}(z)} \varphi_{\alpha,\eps}(d(y,z)) \, d\Leb(y).
\end{multline}
The second summand in~\eqref{eq:phi-integral-bound} does not exceed $\vol(M)\cdot \varphi_{\alpha}(r_0)$, as
\[
\varphi_{\alpha}(r)=\lim_{\eps\to +0} \varphi_{\alpha,\eps}(r) = r^{-\frac{k+\alpha}{2}}
\]
is an upper bound for $\varphi_{\alpha,\eps}(r)$ for all $\eps>0$. The first one can be estimated uniformly in~$z\in M$ by passage to the geodesic coordinates centred at~$z$, comparing it to the same integral in a ball in $\R^k$: the Jacobian of the change of variables is (for $r_0$ smaller than the injectivity radius in $M$) uniformly bounded, and the integral of $r^{-\frac{k+\alpha}{2}}$ on $B_{r_0}(0)\subset \R^k$ converges.

We thus have
\begin{equation}\label{eq:C-K-alpha}
\bK_{\alpha,\eps}^{(r_0)}(x,z) \le C^K_{\alpha}(r_0)
\end{equation}
for some constant $C^K_{\alpha}(r_0)$.

Now, let us transform the first integral, $K_{\alpha,\eps}^{(r_0)}(x,z)$. For sufficiently small $r_0>0$ (smaller than the injectivity radius at every point) one can take the geodesic coordinates at any point $x\in M$ in a ball of radius $r_0$. Thus, we can take a point $x'\in \R^k$, let $\psi:B_{r_0}(x')\to B_{r_0}(x)$ be geodesic coordinates, and denote $z':=\psi^{-1}(z)$.

Fix any $\delta_1>0$. Due to the compactness of $M$, for a sufficiently small $r_0$ the Jacobian and the bi-Lipschitz constants of $\psi$ are $\delta_1$-close to~$1$:
\[
\fL(\psi) < 1+\delta_1, \quad \Jac(\psi) \in \left(\frac{1}{1+\delta_1}, 1+\delta_1\right).
\]
Making a change of variables $y=\psi(y')$ in the integral for~$K_{\alpha,\eps}^{(r_0)}(x,z)$, we get
\begin{multline*}
K_{\alpha,\eps}^{(r_0)}(x,z) = \int_{B_{r_0}(x)} \varphi_{\alpha,\eps}(d(x,y)) \varphi_{\alpha,\eps}(d(z,y)) \, d\Leb_M(y)   \\
= \int_{B_{r_0}(x')} \varphi_{\alpha,\eps}(d(\psi(x'),\psi(y'))) \varphi_{\alpha,\eps}(d(\psi(z'),\psi(y'))) \, d\Leb_M(\psi(y'));
\end{multline*}
all the three quotients
\[
\frac{\varphi_{\alpha,\eps}(d(\psi(x'),\psi(y')))}{\varphi_{\alpha,\eps}(|x'-y'|)},
\quad
\frac{\varphi_{\alpha,\eps}(d(\psi(z'),\psi(y')))}{\varphi_{\alpha,\eps}(|z'-y'|)},
\quad
\left. \frac{d\, \Leb_M}{d \, \psi_* \Leb_{\R^k}} \right|_{y}
\]
are close to~$1$: the first two by a factor $(1+\delta_1)^{\frac{k+\alpha}{2}}$ due to Proposition~\ref{prop:phifReg}, and the last one by the factor $(1+\delta_1)$. Thus, $K_{\alpha,\eps}^{(r_0)}(x,z)$ differs from
\begin{equation}\label{eq:I-r_0}
I_{\alpha,\eps}^{(r_0)} (x',z') :=\int_{B_{r_0}(x')} \varphi_{\alpha,\eps}(|x'-y'|) \varphi_{\alpha,\eps}(|z'-y'|) \, d\Leb_{\R^k}(y')
\end{equation}
by the factor at most $(1+\delta_1)^{k+\alpha+1}$. Now, $I_{\alpha,\eps}^{(r_0)} (x',z')$ is also a part of $I_{\alpha,\eps} (x',z')$, that differs from it by
\[
\overline{I}_{\alpha,\eps}^{(r_0)} (x',z') :=\int_{\R^k\setminus B_{r_0}(x')} \varphi_{\alpha,\eps}(|x'-y'|) \varphi_{\alpha,\eps}(|z'-y'|) \, d\Leb_{\R^k}(y'),
\]
that is bounded uniformly in $\eps, x', z'$ (for $|x'-z'|<\frac{r_0}{2}$ due to the convergence of the integral of $r^{-k-\alpha}$ at infinity, for $|x'-z'|\ge \frac{r_0}{2}$ due to Lemma~\ref{l:lim-U}). Thus, first choosing $\delta_1$ so that $(1+\delta_1)^{k+\alpha+2}<1+\delta$ and then accordingly choosing $r_0$, we obtain the desired estimate~\eqref{eq:K-U-close}.

\end{proof}

\subsection{$E_{\mgr} \left[ {\tEn_{\alpha, \eps} (f_* \msp)} \right]$ vs $\tEn_{\alpha, \eps} (\msp)$: Proof of Proposition \ref{lm:tailsEst}}\label{ss.proof417}

\begin{proof}[Proof of Proposition \ref{lm:tailsEst}]

Let us start by establishing that an estimate~(\ref{tailsEst1}) holds with~$\tEn_{\alpha, \eps}$ replaced by~$\En_{\alpha, \eps}$. Indeed, from Proposition~\ref{prop:enChangeOneStep} we know that for any $f \in \Diff^1(M)$ one has
    \begin{equation*}
        \fL(f)^{-\alpha} \le \frac{\En_{\alpha, \eps} (f_* \msp)}{\En_{\alpha, \eps} (\msp)} \le \fL(f)^{\alpha}.
    \end{equation*}
Taking the expectation with respect to measure $\mgr$ we get
    \begin{equation*}
        \frac{\mathbb{E}_{\mgr} \left[ \En_{\alpha, \eps} (f_* \msp) \right]}{\En_{\alpha, \eps} (\msp)} \in \left( \mathbb{E} \left[ \fL(f)^{-\alpha} \right], \mathbb{E} \left[ \fL(f)^{\alpha} \right] \right).
    \end{equation*}
    It remains to notice that by H\"older's inequality we have
    \begin{equation*}
        \mathbb{E} \left[ \fL(f)^{\alpha} \right] \le \left( \mathbb{E} \left[ \fL(f)^{\gamma} \right] \right)^{\alpha/\gamma}
    \end{equation*}
    and by Jensen's inequality we also have
    \begin{equation*}
        \left( \mathbb{E} \left[ \fL(f)^{\alpha} \right] \right)^{-1} \le \mathbb{E} \left[ \fL(f)^{-\alpha} \right].
    \end{equation*}
Hence, taking $C_0$ such that $\E_{\mgr} \fL(f)^{\gamma}  < C_0$, one gets that for any $\alpha<\gamma$,
    \begin{equation*}
        \frac{\mathbb{E}_{\mgr} \left[ \En_{\alpha, \eps} (f_* \msp) \right]}{\En_{\alpha, \eps} (\msp)} \in \left( C_0^{-\frac{\alpha}{\gamma}}, C_0^{\frac{\alpha}{\gamma}} \right).
    \end{equation*}
Since $C_0^{\frac{\alpha}{\gamma}}$ tends to~$1$ as $\alpha$ tends to~$0$, for any $\dn>0$ there exists $\alpha_0$ such that for any $\alpha\in (0,\alpha_0)$ one has
    \begin{equation}\label{eq:E-f-power}
        \frac{\mathbb{E}_{\mgr} \left[ \En_{\alpha, \eps} (f_* \msp) \right]}{\En_{\alpha, \eps} (\msp)} \in \left( \frac{1}{1+\dn}, 1+\dn \right);
    \end{equation}

Let us now return back to the original~\eqref{tailsEst1}. 
Namely,  let $\delta>0$ be given; choose and fix $\dn>0$ such that $(1+\dn)^3<1+\delta$, and let $\alpha_0$ be chosen w.r.t. $\dn$ so that~\eqref{eq:E-f-power} holds.


By Proposition~\ref{prop:EnEquiv},  
 there exists $\Cn>0$ such that for any measure $\msp'$ on $M$ and any $\eps>0$ one has
\begin{equation}\label{eq:nu-arb}
\frac{1}{1+\dn} \, \En_{\alpha,\eps}(\msp') - \frac{\Cn}{1+\dn}  <
\tEn_{\alpha,\eps}(\msp')  < (1+\dn) \En_{\alpha,\eps}(\msp') + \Cn
\end{equation}
Applying this for $\msp'=f_*\msp$ and taking the expectation w.r.t. $\mgr$ provides
\begin{multline*}
\frac{1}{1+\dn} \, \E_{\mu}\left [\En_{\alpha,\eps}(f_*\msp) \right] - \frac{\Cn}{1+\dn}  <
 \E_{\mu}\left [\tEn_{\alpha,\eps}(f_*\msp) \right]
\\
< (1+\dn) \E_{\mu}\left [\En_{\alpha,\eps}(f_*\msp) \right] + \Cn;
\end{multline*}
using~\eqref{eq:E-f-power}, we thus get
\begin{equation*}
\frac{1}{(1+\dn)^2} \, \En_{\alpha,\eps}(\msp) - \frac{\Cn}{1+\dn}  <
 \E_{\mu}\left [\tEn_{\alpha,\eps}(f_*\msp) \right]
\\
< (1+\dn)^2 \En_{\alpha,\eps}(\msp)  + \Cn.
\end{equation*}
Finally, using~\eqref{eq:nu-arb} with $\msp'=\msp$ to estimate $\En_{\alpha,\eps}$ via $\tEn_{\alpha,\eps}$, we get
\begin{multline*}
\frac{1}{(1+\dn)^2} \cdot \frac{1}{(1+\dn)}\left(\En_{\alpha,\eps}(\msp) - \Cn\right) - \frac{\Cn}{1+\dn}  <
 \E_{\mu}\left [\tEn_{\alpha,\eps}(f_*\msp) \right]
\\
< (1+\dn)^2 \cdot \left ((1+\dn) \En_{\alpha,\eps}(\msp) + \Cn\right)  + \Cn;
\end{multline*}
as $(1+\dn)^3<1+\delta$, we get the desired
\[
(1+\delta) (\En_{\alpha,\eps}(\msp) - C') < \E_{\mu}\left [\tEn_{\alpha,\eps}(f_*\msp) \right] <
(1+\delta) \En_{\alpha,\eps}(\msp) + C'
\]
for some constant~$C'$. As $\delta>0$ was arbitrary, and taking    (\ref{eq:nu-arb})    into account again, \eqref{tailsEst1} follows. 
\end{proof}

\subsection{
Estimating $W(f_* \mTn_{\alpha, \eps} (\msp), \mTn_{\alpha, \eps} (f_* \msp))$: 
Proof of Proposition~\ref{prop:WassEst}}\label{s:WE}

\begin{proof}[Proof of Proposition ~\ref{prop:WassEst}] For given $\alpha,\eps>0$ and measure $\msp$, consider a (non-probability) measure $\dm_{\alpha,\eps}(\msp)$ on $M\times M \times M$, given by
\[
\dm_{\alpha,\eps}(\msp)= \varphi_{\alpha,\eps}(d(x,y)) \varphi_{\alpha,\eps}(d(z,y)) \, \msp(dx) \, \Leb(dy) \, \msp(dz).
\]

Denote by $\pi_1$, $\pi_2$ the projections of $M\times M\times M$ on the first and second coordinate respectively, and by $\pi_{1,3}$ the projection on $M\times M$ corresponding to the first and third coordinates. Then directly by definition
\[
(\pi_{2})_* \, \dm_{\alpha,\eps}(\msp) = \mT_{\alpha,\eps}(\msp);
\]
also, define
\begin{equation}\label{eq:pi-1-3}
\mTt_{\alpha,\eps}(\msp):= (\pi_{1,3})_* \, \dm_{\alpha,\eps}(\msp) = K_{\alpha,\eps}(x,z) \, \msp(dx) \, \msp(dz);
\end{equation}
the second equality is due to~\eqref{eq:K-a-e}.
Finally, consider the measure
\begin{equation}
\mTf_{\alpha,\eps}(\msp):=(\pi_1)_* \, \dm_{\alpha,\eps}(\msp)
\end{equation}
as well as the normalizations of these measures,
\begin{equation}
\mTtn_{\alpha,\eps}(\msp):=\frac{1}{\tEn_{\alpha,\eps}(\msp)}  \mTt_{\alpha,\eps}(\msp), \quad
\mTfn_{\alpha,\eps}(\msp):=\frac{1}{\tEn_{\alpha,\eps}(\msp)} \mTf_{\alpha,\eps}(\msp).
\end{equation}


 For a high-energy measure $\msp$ most of the measure $\dm_{\alpha,\eps}(\msp)$ is concentrated near the diagonal, and hence the projections on the first and on the second coordinates are close to each other. The following lemma formalizes this argument:

\begin{lemma}\label{l:W-close}
For any $\delta_1$, $\alpha$ there exists $C'$ such that for any $\eps>0$ and $\msp$ with $\tEn_{\alpha,\eps}(\msp)>C'$ one has
\[
W(\mTn_{\alpha,\eps}(\msp),  \mTfn_{\alpha,\eps}(\msp))<\delta_1.
\]
\end{lemma}
\begin{proof}
Take $r_0:=\frac{\delta_1}{2}$ and let
\[
A_{r_0}:=\{(x,y,z)\in M^3 \mid d(x,y)< r_0 \}, \bA_{r_0}:=M^3 \setminus A_{r_0}.
\]
From the proof of Lemma~\ref{l:K-U-comp}, we have
\[
\dm_{\alpha,\eps}(\msp)(\bA_{r_0}) = \int \bK_{\alpha,\eps}^{(r_0)}(x,z) \, d\msp(x) \, d\msp(z) \le C^K_{\alpha}(r_0),
\]
where the second inequality is due to~\eqref{eq:C-K-alpha}. Thus, the non-normalized measure $\dm_{\alpha,\eps}(\msp)(\bA_{r_0})$ does not exceed a constant~$C^K_{\alpha}(r_0)$.

Now, we can couple the normalized measures $
\mTn_{\alpha,\eps}(\msp)$ and 
$ \mTfn_{\alpha,\eps}(\msp)$
using the projection of $\frac{1}{\tEn_{\alpha,\eps}(\msp)} \dm_{\alpha,\eps}(\msp)$ on the first two coordinates; this coupling leads to the upper bound for the Wasserstein distance
\begin{multline}\label{eq:W-delta-1}
W(\mTn_{\alpha,\eps}(\msp), \mTfn_{\alpha,\eps}(\msp))\le r_0 \cdot \frac{\dm_{\alpha,\eps}(\msp)(A_{r_0})}{\tEn_{\alpha,\eps}(\msp)}
\\
+ \diam(M) \cdot \frac{\dm_{\alpha,\eps}(\msp)(\bA_{r_0})}{\tEn_{\alpha,\eps}(\msp)}    \le r_0+ \frac{C^K_{\alpha}(r_0)}{\tEn_{\alpha,\eps}(\msp)},
\end{multline}
where the first summand corresponds to the points with $d(x,y)<r_0$, and the second to the points with $d(x,y)\ge r_0$. As we chose $r_0=\frac{\delta_1}{2}$, it suffices to require that
\[
\tEn_{\alpha,\eps}(\msp) > \frac{2 \, C^K_{\alpha}(r_0)}{\delta_1} =: C'
\]
to ensure that the total Wasserstein distance does not exceed $\delta_1$.
\end{proof}

On the other hand, for the measures $\mTfn_{\alpha,\eps}(\msp)$ the analogue of Proposition~\ref{prop:WassEst} can be established directly, and one can even estimate the total variations distance:
\begin{lemma}\label{l:TV-prime}
For any $\delta_2 > 0$ and any $R > 0$ there exists $\alpha_2 > 0$ such that for any $\alpha \in (0,\alpha_2)$ there exists $C > 0$ such that for any $f \in \Diff^1(M)$ with $\fL(f) < R$, any $\eps > 0$, and any $\msp$ such that $\tEn_{\alpha, \eps} (\msp) > C$ or $\En_{\alpha, \eps} (\msp) > C$  one has
\begin{equation}\label{eq:TV-theta-prime}
\TV(f_* \mTfn_{\alpha,\eps}(\msp),\mTfn_{\alpha,\eps}(f_* \msp)) < \delta_2,
\end{equation}
and, hence,
\begin{equation}\label{eq:W-TV}
W(f_* \mTfn_{\alpha,\eps}(\msp),\mTfn_{\alpha,\eps}(f_* \msp)) < \delta_2 \cdot \diam(M).
\end{equation}
\end{lemma}
\begin{proof}
We will first use Lemma~\ref{l:K-U-comp} together with Proposition~\ref{prop:phifReg}
to compare $K_{\alpha,\eps}(x,z)$ with $K_{\alpha,\eps}(f(x),f(z))$. Namely, choose $\alpha_0$ and $\delta_3$ so small that
\[
(1+\delta_3)^2 \cdot R^{\alpha_0} < 1+\frac{\delta_2}{10}.
\]
Then from Lemma~\ref{l:K-U-comp} we know that there exists $C>0$ such that for any $x,z,\eps$ we have
\[
K_{\alpha,\eps}(x,z) \approx_{(\delta_3,C)} U_{\alpha,\eps}(x,z), \quad  K_{\alpha,\eps}(f(x),f(z)) \approx_{(\delta_3,C)} U_{\alpha,\eps}(f(x),f(z)),
\]
while from~\eqref{UReg} in Proposition~\ref{prop:phifReg}  we get
\[
\frac{U_{\alpha,\eps}(f(x),f(z))}{U_{\alpha,\eps}(x,z)} \in (\fL(f)^{-\alpha}, \fL(f)^{\alpha}) \subset (R^{-\alpha_0},R^{\alpha_0}).
\]
Joining these three estimates together, we obtain
\begin{equation}\label{eq:K-delta-C}
K_{\alpha,\eps}(x,z) \approx_{(\frac{\delta_2}{10},C'')}K_{\alpha,\eps}(f(x),f(z))
\end{equation}
for some explicit constant~$C''$.

Now, applying $f^{-1}_*$ does not change the total variations distance, so instead of~\eqref{eq:TV-theta-prime} we can show the equivalent statement
\begin{equation}\label{eq:d-2}
\TV(\mTfn_{\alpha,\eps}(\msp),f^{-1}_* \mTfn_{\alpha,\eps}(f_* \msp)) < \delta_2.
\end{equation}
The measures here can be obtained as projections of measures on $M\times M$:
\[
\mTfn_{\alpha,\eps}(\msp) =  (\pi_1)_* \, \mTtn_{\alpha,\eps}(\msp),
\]
\[
f^{-1}_* \mTfn_{\alpha,\eps}(f_* \msp) =
 (\pi_1)_* \left(f^{-1}_* \mTtn_{\alpha,\eps}(f_* \msp)\right)
\]
where $\pi_1$ is the projection on the first coordinate. To obtain the desired~\eqref{eq:d-2}, we will actually show that for sufficiently high energy~$\tEn_{\alpha,\eps}(\msp)$ even before projection one has
\begin{equation}\label{eq:TV-delta-2}
\TV(\mTtn_{\alpha,\eps}(\msp),f^{-1}_* \mTtn_{\alpha,\eps}(f_* \msp))<\delta_2.
\end{equation}

To do so, note that both these measures
are absolutely continuous with respect to
$\nu\times \nu$:
\[
\mTtn_{\alpha,\eps}(\msp) = \frac{1}{\tEn_{\alpha,\eps}(\msp)}  K_{\alpha,\eps}(x,z) \, \nu(dx) \, \nu (dz),
\]
\[
f^{-1}_* \mTtn_{\alpha,\eps}(f_* \msp) =  \frac{1}{\tEn_{\alpha,\eps}(f_* \msp)} K_{\alpha,\eps}(f(x),f(z)) \, \nu(dx) \, \nu (dz).
\]
Now, \eqref{eq:K-delta-C} implies that once $K_{\alpha,\eps}(x,z)>C_2=:\frac{20C''}{\delta_2}$, the $K_{\alpha,\eps}$ parts of these densities are close to each other:
\[
\frac{1}{1+\frac{\delta_2}{6}}< \frac{K_{\alpha,\eps}(f(x),f(z))}{K_{\alpha,\eps}(x,z)} < 1+\frac{\delta_2}{6}.
\]
On the other hand, due to Corollary~\ref{cor:tEnChangeOneStep}, once the energy $\tEn_{\alpha,\eps}(\msp)$ is sufficiently large, the normalization constants are also close to each other,
\[
\frac{1}{1+\frac{\delta_2}{6}} \, \tEn_{\alpha,\eps}(\msp)<\tEn_{\alpha,\eps}(f_*\msp) <
(1+\frac{\delta_2}{6}) \, \tEn_{\alpha,\eps}(\msp).
\]
Multiplying the two inequalities, we get that on the set $$A:=\{(x,z)\mid K_{\alpha,\eps}(x,z)>C_2\}$$ the quotient of densities w.r.t. $\msp\times \msp$ of the two measures is in the interval $\left(\frac{1}{(1+\frac{\delta_2}{6})^2}, (1+\frac{\delta_2}{6})^2\right)$.  

Finally, the $\mTtn_{\alpha,\eps}(\msp)$-measure of its complement does not exceed $\frac{C_2}{\tEn_{\alpha,\eps}(\msp)}$. Hence, if the energy $\tEn_{\alpha,\eps}(\msp)$ is sufficiently high to make sure that $\frac{C_2}{\tEn_{\alpha,\eps}(\msp)}<\frac{\delta_2}{6}$, the part that the normalized measures have in common is at least
\[
\frac{1-\frac{\delta_2}{6}}{(1+\frac{\delta_2}{6})^2} > 1-\frac{\delta_2}{2},
\]
and hence the total variation~\eqref{eq:TV-delta-2} indeed does not exceed~$\delta_2$.

An application of Lemma~\ref{lm:WasEst} concludes the proof of the upper bound~\eqref{eq:W-TV} for the Wasserstein distance.

%
%
%
%
%
%
%
%
\end{proof}

Lemma \ref{l:W-close}  and   Lemma \ref{l:TV-prime}  
together imply Proposition~\ref{prop:WassEst}. Indeed, we have:
\begin{multline*}
W(f_* \mTn_{\alpha,\eps}(\msp), \mTn_{\alpha,\eps}(f_* \msp) )
\le
W(f_* \mTn_{\alpha,\eps}(\msp), f_* \mTfn_{\alpha,\eps}(\msp))+
\\
+ W(f_* \mTfn_{\alpha,\eps}(\msp),\mTfn_{\alpha,\eps}(f_* \msp))
+ W(\mTfn_{\alpha,\eps}(f_* \msp), \mTn_{\alpha,\eps}(f_* \msp))
\end{multline*}
The first and the third summands can be estimated directly using Lemma~\ref{l:W-close}. Indeed, for the first summand, it suffices to note that the application of~$f$ increases the Wasserstein distance at most $\fL(f)<R$ times; for the last summand, to apply this lemma, we note that an upper bound on $\fL(f)$ implies that $f_*\msp$ is of sufficiently high energy provided that $\msp$ is of high energy.


Finally, the second summand is estimated directly by Lemma~\ref{l:TV-prime}.
\end{proof}

\subsection{Estimating $W(\mTn_{\alpha, \eps} (f_* \msp), \mTn_{\alpha, \eps}(\mgr * \msp))$: Proof of Lemma \ref{l.West}}\label{s:W-proof}
    \begin{proof}[Proof of Lemma \ref{l.West}]
    	We start by working with the non-normalized measures $\mT_{\alpha,\eps}(f_*\msp), \mT_{\alpha,\eps}(\mgr*\msp)$. To simplify the notation, denote
\[
	g_f(y):=\rho_{\alpha,\eps}[f_*\msp](y), \quad \bg(y):=\rho_{\alpha,\eps}[\mgr*\msp](y).
\]
Therefore, we have
$$
\mT_{\alpha,\eps}[f_*\msp]=(g_f(y))^2d\text{Leb}(y)\ \ \text{\rm and}\ \ \mT_{\alpha,\eps}[\mgr*\msp]=(\bg(y))^2d\text{Leb}(y).
$$
From~\eqref{contrAssumption2} we have
\[
\tEn_{\alpha, \eps} (\msp) \le \frac{1}{1-\delta} \tEn_{\alpha, \eps} (\mgr * \msp).
\]
Joining with~\eqref{eq:exp-rho}, we have
\begin{equation}\label{eq:exp-rho-f}
\E_f \left[ \int_M |g_f (y) - \bg (y) |^2 \, d\Leb(y)\right] \le \frac{2\delta}{1-\delta} \tEn_{\alpha, \eps} (\mgr * \msp).
\end{equation}
By Markov inequality, with the probability at least $1-\frac{2\sqrt{\delta}}{1-\delta}$, one has
\begin{multline}\label{eq:f-bar}
\int_M |g_f (y) - \bg (y) |^2 \, d\Leb(y) \le \sqrt{\delta}\,  \tEn_{\alpha, \eps} (\mgr * \msp)
\\
= \sqrt{\delta} \, \mT_{\alpha, \eps} [\mgr* \msp] (M)
= \sqrt{\delta} \int_M \bg^2(y) \, d\Leb(y).
\end{multline}
Now, for any $f$ for which~\eqref{eq:f-bar} holds, consider the set
\[
X_f:=\{y\in M \, : \, |g_f (y) - \bg (y) | \le \sqrt[8]{\delta} \cdot \bg(y)\}.
\]
Again from Markov inequality type argument, one has
\begin{equation}\label{eq:M-Xf}
\mT_{\alpha, \eps} [\mgr* \msp] (X_f) \ge \left( 1-\sqrt[4]{\delta}\right) \mT_{\alpha, \eps} [\mgr* \msp] (M).
\end{equation}
Indeed, on the complement $M\setminus X_f$ one has $|g_f (y) - \bg (y) |^2> \sqrt[4]{\delta} \cdot \bg^2(y)$; integrating, one gets
\[
\int_{M\setminus X_f} |g_f (y) - \bg (y) |^2 \, d\Leb(y) > \sqrt[4]{\delta} \cdot \mT_{\alpha, \eps} [\mgr* \msp] (M\setminus X_f).
\]
Thus, if~\eqref{eq:M-Xf} did not hold, it would imply
\[
\mT_{\alpha, \eps} [\mgr* \msp] (M\setminus X_f) > \sqrt[4]{\delta} \cdot  \mT_{\alpha, \eps} [\mgr* \msp] (M),
\]
hence providing a contradiction with~\eqref{eq:f-bar}.


Now,~\eqref{eq:M-Xf} implies that the non-normalized measures $\mT_{\alpha,\eps}[f_*\msp]$ and $\mT_{\alpha,\eps}[\mgr*\msp]$ share a common part
\[
\Xi_f:=(1-\sqrt[8]{\delta})^2 \, \bg^2(y) \cdot \Ind_{X_f}(y) \, d\Leb(y)
\]
of measure at least
\begin{equation}\label{eq:Xi}
\Xi_f(M)\ge (1-\sqrt[8]{\delta})^2 (1-\sqrt[4]{\delta}) \cdot \mT_{\alpha, \eps} [\mgr* \msp] (M).
\end{equation}

Finally, note that the normalization constants are also close with high probability, namely, for $f$ such that the inequality~\eqref{eq:f-bar} holds. Indeed,  the inequality~\eqref{eq:f-bar}  can be rewritten as
\[
\|g_f-\bg \|_{L_2(M)}^2 \le \sqrt{\delta} \cdot \|\bg \|_{L_2(M)}^2;
\]
hence, for any $f$ for which it holds, the triangle inequality implies
\begin{equation}\label{eq:mTf}
\mT_{\alpha, \eps} [f_* \msp] (M) =
\|g_f\|_{L_2(M)}^2 \le (1+\sqrt[4]{\delta})^2
\mT_{\alpha, \eps} [\mgr* \msp] (M).
\end{equation}
Now, the normalized measures $\mTn_{\alpha, \eps} [f_* \msp]$ and $\mTn_{\alpha, \eps} [\mgr * \msp]$
share the common part $\frac{1}{C_f} \Xi_f$, where
\[
C_f=\max(\mT_{\alpha, \eps} [f_* \msp](M), \mT_{\alpha, \eps} [\mgr * \msp](M))
\]
is the maximum of two normalization constants. Using~\eqref{eq:Xi} and~\eqref{eq:mTf}, we see that this part has measure at least
\[
\frac{1}{C_f} \Xi_f(M) \ge \frac{(1-\sqrt[8]{\delta})^2 (1-\sqrt[4]{\delta})}{(1+\sqrt[4]{\delta})^2} \ge 1-10\sqrt[8]{\delta}.
\]
Hence, for any $f$ satisfying~\eqref{eq:f-bar} the total variation distance between the normalized measures does not exceed
\[
\TV(\mTn_{\alpha, \eps} (f_* \msp), \mTn_{\alpha, \eps}(\mgr * \msp))) \le 10\sqrt[8]{\delta}.
\]
As~\eqref{eq:f-bar} holds with the probability at least~$1-\frac{2\sqrt{\delta}}{1-\delta}> 1-4\sqrt{\delta}$, this establishes conclusion~\eqref{eq:Probably-TV}. Finally, this immediately implies conclusion~\eqref{eq:Probably} due to Lemma~\ref{lm:WasEst}.

\end{proof}


\begin{appendix}\label{a.1}
\section{Example}\label{ss.example}

Notice that our main result is applicable to the case when all the maps we consider are projective maps of $\mathbb{RP}^d$. H\"older continuity of a stationary measure in that case is a well known result by Guivarc'h \cite{G}, as we discussed in Section \ref{ss.smrmp} above. The condition on finiteness of moments (used both in our main result and in \cite{G}) is more restrictive than the condition in the Furstenberg Theorem on random matrix products that ensures positivity of Lyapunov exponent, i.e. $\mathbb{E}\log\|A\|<\infty$, see \cite{Fur1}. Let us give an example that shows that this is not an artifact of the proof, and that the original Furstenberg condition is not sufficient to guarantee H\"older continuity of stationary measures.

\begin{proposition}\label{p.example}
There exists a probability distribution $\mgr$ on $\SL(2, \mathbb{R})$ such that

\vspace{5pt}

1) The smallest subgroup generated by $\text{supp}\,\mgr$ is proximal (in particular, not compact) and strongly irreducible;

\vspace{5pt}

2) $\mathbb{E}_{\mgr}\log\|A\|<\infty$;

\vspace{5pt}

3) For any $\gamma>0$ we have $\mathbb{E}_{\mgr}\|A\|^\gamma =\infty$;

\vspace{5pt}

4) Any stationary measure on $\mathbb{RP}^1$ is not H\"older continuous.
\end{proposition}

\begin{proof}
For each $n\in \mathbb{N}$ set  $$A_n=R_{\frac{1}{5n}}\cdot
\begin{pmatrix}
2^{n^2} & 0 \\
0 & 2^{-n^2} \\
\end{pmatrix}\cdot R_{-\frac{1}{5n}},$$ where $R_{\alpha}=\begin{pmatrix}
                                                   \cos 2\pi \alpha & \sin 2\pi \alpha\\
                                                   -\sin 2\pi \alpha & \cos 2\pi \alpha \\
                                                 \end{pmatrix}$, and define the distribution $\mgr$ on $\SL(2, \mathbb{R})$ by
$$
\mgr(A_n)=p_n\equiv \frac{1}{2^n}.
$$
It is straightforward to check that the conditions 1), 2), and 3) of the Proposition \ref{p.example} are satisfied. Let us show that 4) also holds. Let us denote by $f_n:\mathbb{RP}^1\to \mathbb{RP}^1$ a projectivization of $A_n:\mathbb{R}^2\to \mathbb{R}^2$. Slightly abusing the notation, let us use the same symbol $\mgr$ also for the distribution on the space of projective maps with $\mgr(f_n)=p_n=\frac{1}{2^n}$. Each map $f_n$ has an attracting fixed point that we denote by $s_n$. Let $\msp$ be a stationary probability measure on $\mathbb{RP}^1$. Due to the explicit form of the matrices $A_n$ we have $\text{supp}\, \msp\subseteq [0, 2/5]$ (here we identify $\mathbb{RP}^1$ with $[0,1]/_{\{0\sim 1\}}$). If $x\in [0, 2/5]\subset \mathbb{RP}^1$ and $v\in \mathbb{R}^2$ is the corresponding unit vector, then
$$
|A_nv|\ge 2^{n^2}\cos \frac{2\pi}{5}>\frac{2^{n^2}}{10},\ \text{and hence}\ |f'_n(x)|=\frac{1}{|A_nv|^2}<\frac{100}{2^{{2n^2}}},
$$
Notice that with probability $p_n^k$ we can apply the matrix $A_n$ in the random product $k$ times in a row. Since
$$
f_n^k(\supp\,\msp)\subset f_n^k([0, 2/5])\subset \left[s_n-\left(\frac{100}{2^{2n^2}}\right)^k, s_n+\left(\frac{100}{2^{2n^2}}\right)^k\right]\equiv I_{n,k},
$$
we have $\msp(I_{n,k})\ge p_n^k=\frac{1}{2^{kn}}$. This implies that
$$
\frac{\log \msp(I_{n,k})}{\log|I_{n,k}|}\le \frac{-n\log 2}{\frac{\log 2}{k}+\log 100-2n^2\log 2}\to 0\ \ \text{as}\ \ n\to \infty,
$$
hence the measure $\msp$ cannot be H\"older continuous.
\end{proof}
\end{appendix}

\section*{Acknowledgments}
The authors are grateful to B. Barany, S. Cantat, B. Deroin, Y. Guivarc'h, F. Ledrappier, and B.\,Solomyak for fruitful discussions and for providing helpful references.

\end{document}